\documentclass[10pt,reqno,a4paper]{amsart}

\usepackage{amsmath,amssymb,amsthm}
\usepackage{graphicx,xcolor}
\usepackage{eucal}
\usepackage{enumitem}
\usepackage{mathrsfs}

\newtheorem{theorem}{Theorem}[section]
\newtheorem{lemma}[theorem]{Lemma}
\newtheorem{proposition}[theorem]{Proposition}

\theoremstyle{remark}
\newtheorem{remark}[theorem]{Remark}
\theoremstyle{definition}

\newtheorem{example}[theorem]{Example}

\newtheorem{ass}{Assumption}

\newtheorem{assu}{Assumption}

\numberwithin{equation}{section}

\newcommand{\itv}{(\alpha, \beta)}
\newcommand{\citv}{[\alpha,\beta)}
\newcommand{\itvla}{(t_\la,\beta)}
\newcommand{\itvlat}{(\widetilde t_\la,\beta)}
\newcommand{\citvla}{[t_\la,\beta)}
\newcommand{\citvlatil}{[\widehat t_\la,\beta)}
\newcommand\dist{\operatorname{dist}}

\newcommand\la{\lambda}

\newcommand\de{\delta}

\newcommand\sa{\sigma}

\newcommand\Dis{{\mathfrak D}}

\newcommand\cA{{\mathcal A}}
\newcommand\cD{{\mathcal D}}
\newcommand\cH{{\mathcal H}}
\newcommand\cR{{\mathcal R}}

\newcommand{\N}{{\mathbb N}}
\newcommand{\RR}{{\mathbb R}}
\newcommand{\CC}{{\mathbb C}}

\newcommand\pt{{\displaystyle{\frac{\partial}{\partial t}}}}
\newcommand\dt{{\displaystyle{\frac{\d}{\d t}}}}
\newcommand\dx{{\displaystyle{\frac{\d}{\d x}}}}

\DeclareMathOperator\Real{Re}
\DeclareMathOperator\Imag{Im}
\renewcommand\Re{\Real}
\renewcommand\Im{\Imag}
\renewcommand\d{{\rm d}}
\newcommand{\e}{{\rm e}}
\newcommand{\ii}{{\rm i}}
\renewcommand\o{{\rm o}}
\renewcommand\O{{\rm O}}

\DeclareMathOperator\ran{ran}

\newcommand\ds{\displaystyle}

\newcommand\wt{\widetilde}
\newcommand\wh{\widehat}
\newcommand{\ov}[1]{\overline{#1}}
\newcommand\mydot{\,\cdot\,}
\newcommand{\defeq}{\mathrel{\mathop:}=}
\newcommand\spt{\sigma_{\rm p}}
\newcommand\sess{\sigma_{\rm ess}}
\newcommand\sessreg{\sigma_{\rm ess}^{\rm \,r}}
\newcommand\sesssing{\sigma_{\rm ess}^{\rm \,s}}

\newcommand\tpi{{\pi}}

\newcommand\tka{{\varkappa}}

\newcommand{\sgn}{\text{sgn}}

\begin{document}


\title[Essential spectrum of matrix differential operators]{Analysis of the essential spectrum of \\ singular matrix differential operators}

\author{O.\,O.\ Ibrogimov}
\address[O.\,O.\ Ibrogimov]{%
	Mathematisches Institut, 
	Universit\"{a}t Bern,
	Alpeneggstr.\ 22,
	3012 Bern, Switzerland}
\email{orif.ibrogimov@math.unibe.ch}

\author{P.\ Siegl}
\address[P.\ Siegl]{%
	Mathematisches Institut, 
	Universit\"{a}t Bern,
	Alpeneggstr.\ 22,
	3012 Bern, Switzerland
	\& on leave from Nuclear Physics Institute ASCR, 25068 \v Re\v z, Czech Republic}
\email{petr.siegl@math.unibe.ch}

\author{C.\ Tretter}
\address[C.\ Tretter]{%
	Mathematisches Institut, 
	Universit\"{a}t Bern,
	Sidlerstr.\  5,
	3012 Bern, Switzerland
	\& Matematiska institutionen, Stockholms universitet, 106 91 Stockholm,~Sweden}
\email{tretter@math.unibe.ch}

\date{\today}

\dedicatory{Dedicated to Professor Heinz Langer on the occasion of his 80th birthday}

\begin{abstract}
\!A complete analysis of the essential spectrum of matrix-differential operators $\cA$ of the form
\begin{equation}
\label{mo}
\begin{pmatrix} 
-\dt p \dt + q & -\dt b^* \! + c^* \\[2mm]
\hspace{6mm} b \dt + c & \hspace{4mm} D
\end{pmatrix}
\quad \text{in } \ L^2(\itv) \oplus \bigl(L^2(\itv)\bigr)^n
\end{equation} 
singular at $\beta\in\RR\cup\{\infty\}$ is given; the coefficient functions $p$, $q$ are scalar real-valued with $p>0$, $b$, $c$ are vector-valued, and $D$ is Hermitian matrix-valued.
The so-called ``singular part of the essential spectrum'' $\sesssing(\cA)$ is investigated systematically. Our main results include an explicit description of $\sesssing(\cA)$, criteria for its absence and presence;
an analysis of its topological structure and of the essential spectral radius. Our key tools are: the asymptotics of the leading coefficient $\tpi(\cdot,\la)=p-b^*(D-\la)^{-1}b$ 
of the first Schur complement of \eqref{mo}, a scalar differential operator but non-linear in $\la$;
the Nevanlinna behaviour in $\la$ of certain limits $t\!\nearrow\!\beta$ of functions formed out of the coefficients in~\eqref{mo}.
The efficacy of our results is demonstrated by several applications; in particular, we prove a conjecture on the essential spectrum of some symmetric stellar equilibrium models.
\end{abstract}

\maketitle

\section{Introduction}
The interesting spectral phenomena of matrix differential operators have attracted a lot of attention in recent years. In particular, the essential spectrum and 
the different mechanisms giving rise to it were studied in many papers, see e.g.\ \cite{FMM00}, \cite{MNT02}, \cite{KN03}, \cite{MT07}, \cite{KLN08}, \cite{QCH11}. Often the motivation for 
the particular examples studied therein came from mathematical physics, in particular, magnetohydrodynamics, see \cite{Kako87}. The first paper where the essential 
spectrum of general singular matrix differential operators of the form \eqref{mo} with \emph{scalar} function $D$~was analysed and described explicitly is the 
recent work \cite{ILLT13}. Nevertheless, the results therein did not provide the full solution for the essential spectrum of a problem in symmetric stellar equilibrium models; for the essential spectrum due to the singularity at the boundary of the star only a conjecture was~made.

Here we undertake a systematic analysis of the essential spectrum of matrix differential operators \eqref{mo} 
for the case of matrix-valued $D$.
Under considerably weaker assumptions than in \cite{ILLT13}, we give an explicit description of the part of the essential spectrum caused by the singularity at $\beta$ which we call 
\textit{singular part of the essential spectrum} and which we denote by $\sesssing(\cA)$.
Furthermore, we establish criteria that allow us to give a complete classification and characterization of $\sesssing(\cA)$ and we investigate its topological structure. Our explicit characterization allows to decide when the essential spectrum is bounded and, in this case, to derive a formula for, or estimate, the essential spectral radius 
$r_{\rm{ess}}(\cA) := \sup\, \{|\la|: \la\in\sess(\cA)\}$, cf.\ \cite{Nussbaum70}. Results and particular assumptions
in earlier papers on examples of operators \eqref{mo} are shown to be special cases of our abstract classification. 
Moreover, our weaker assumptions allow us to prove the conjecture in \cite{ILLT13} on the singular part of the essential spectrum for the symmetric stellar equilibrium model from \cite{BEY95}.

The novelty of this paper is that it characterizes all parts and features of the essential spectrum of general singular matrix differential operators \eqref{mo} 
in terms of the coefficients of the associated first Schur complement $S(\la)$ which is a scalar differential operator 
defined for, and depending non-linearly on, $\la \notin \sigma(D)$:
\begin{align}
\label{sc}
S(\la) &= -\dt p \dt \!+\! q \!-\! \la \!-\! \left(- \dt  b^* \!\! +\! c^* \right)(D-\la)^{-1} \left(b \dt \!+\! c \right ) 
\quad \text{in } \ L^2(\itv) \\
\label{sc1}
&=:- \pt \tpi(\cdot,\la)\pt + \ii\Bigl(r(\cdot,\la)\pt + \pt r(\cdot,\la)\Bigr) +   \tka (\cdot,\la);
\end{align}
see \eqref{pi.r.varkap.def} for the precise form of the coefficient functions. 
It turns out that a particular role is played by the leading coefficient $\tpi(t,\la) := p(t)-b(t)^*(D(t)-\la)^{-1}b(t)$, $t\in [\alpha,\beta)$, of $S(\la)$ and by its asymptotic properties:

1. The \emph{regular part} $\sessreg(\cA)$ of the essential spectrum, defined as the closure of the union of the essential spectra of the restrictions $\cA_{[\alpha,\beta_r]}$ of $\cA$ to \emph{regular} subintervals $[\alpha,\beta_r]\subset [\alpha,\beta)$, is identified with the points $\la$ for which $\tpi(\cdot,\la)$ has a zero in $[\alpha,\beta_r]$.

2. The \emph{singular part} $\sesssing(\cA)$ of the essential spectrum, defined as the part that is not present for any restriction $\cA_{[\alpha,\beta_r]}$ to a regular subinterval, is described in terms of 
limits $t\nearrow \beta$ of some functions formed out of the coefficients of the Schur complement; 
e.g.\ if $\beta=\infty$, 
\begin{equation}
\label{sesssing}
  \la \in \sesssing(\cA) \ \iff \ \left( \lim_{t\nearrow\beta}\frac{r(t,\la)}{\tpi(t,\la)} \right)^2 - \lim_{t\nearrow\beta}\frac{\varkappa(t,\la)}{\tpi(t,\la)} \ge 0
\end{equation}
for points $\la\in\RR \setminus (\sessreg(\cA) \cup \Lambda_{\beta}(D))$ where $\Lambda_{\beta}(D) \subset\RR$ is the set of (finite) accumulation points of the eigenvalues of $D(t)$ as $t\nearrow\infty$.

3. The \emph{absence resp.\ presence of the singular part} $\sesssing(\cA)$ of the essential spectrum is fully characterized in terms of the coefficients of 
the  asymptotic expansion of $\tpi(\cdot,\la)$ as $t\nearrow \beta$, 
\begin{equation}
\label{piasy}
\tpi(t,\la) = \tpi_0(\la) + \tpi_1(\la)(t-\beta) + \mathcal{R}(t,\la),
\end{equation}
provided the latter exists and the remainder has certain asymptotic properties; 
more precisely,
\[
 \tpi_0(\la) \neq 0, \ \ \tpi_1(\la) \neq 0 \ \Longrightarrow \ \sesssing(\cA) \setminus  \Lambda_{\beta}(D) = \emptyset. 
\]

4. The \emph{topological structure of} $\sesssing(\cA)$ is classified by means of certain coefficients, some of which naturally arise from the canonical representation of Nevanlinna functions, obtained from the coefficients of the Schur complement such \vspace{-2mm}  as
\begin{alignat*}{2}
g_\beta & = \lim_{\la\to \infty} \frac{-\tpi_2(\la)}{\la}, \quad &
-\tpi_2(\la) &= -\frac 12 \lim_{t\nearrow\beta}\frac{\partial^2}{\partial t^2}\tpi(t,\la), \\[-1.5mm]
\psi_\beta & = \lim_{\la\to \infty} \frac{-\varkappa_0(\la)}{\la}, \quad & 
-\varkappa_0(\la) & =  -\lim_{t\nearrow\beta} \varkappa(t,\la). \vspace{1mm}
\end{alignat*}	
In particular, our classification allows to characterize all cases where the essential spectral radius is finite. 
If e.g.\ $g_\beta > 0$ and $g_\beta + 4 \psi_\beta \ne 0$ and the eigenvalues of $D(t)$ have limits of which $j_0$ are proper, then the closure of the solution set of the inequality in \eqref{sesssing} consists \vspace{1mm} of

\begin{tabular}{ll}
-- at most $j_0+1$ compact intervals if $\,g_{\beta} + 4\psi_{\beta}>0$, \\
-- at most $j_0$ compact intervals and two unbounded intervals 
   if $\,g_{\beta} + 4\psi_{\beta}<0$.
\end{tabular} 

\smallskip

There are three crucial differences compared to earlier papers such as \cite{FMM00}, \cite{MNT02}, \cite{KN03}, \cite{KLN08}, and \cite{Kako87}, \cite{BEY95}. 
First we do not only consider special classes or examples of matrix differential operators \eqref{mo}; secondly we do not only 
consider them under particular assumptions which rule out possibilities for the singular part of the essential spectrum; and thirdly 
our methodology is based on the analysis of Schur complement. The latter allows us to use results from the theory of scalar differential operators and distinguishes our paper also from
the recent paper \cite{QCH11} for the special case of scalar $D$. While \cite{QCH11} shows how to determine $\sess(\cA)$ by a transformation to a Hamiltonian system and relies on a limit-point/circle classification, our method provides an explicit formula for $\sess(\cA)$ in terms of the original coefficient functions and is not restricted 
to the limit-circle case at $\beta$.

There are also essential differences to the earlier paper \cite{ILLT13} in at least three respects. Firstly, we cover~the more general case of 
matrix-valued and not only scalar-valued coefficients $D$. Secondly, and more importantly, we prove all results under much weaker assumptions 
than in \cite{ILLT13}; major improvements include that\vspace{1mm}~e.g.

\begin{tabular}{ll}
-- the eigenvalues of $D(t)$ are no longer assumed to have (proper or improper) limits as $t\nearrow \beta$, \\[1.1mm]
-- $\|(D-\la)^{-1}b\|$ and $\|(D-\la)^{-1}c\|$ no longer need to be bounded near $\beta$,\\[0.4mm]
-- $\tpi(\cdot,\la)$, {\small$\dfrac{1}{\tpi(\cdot,\la)}$}, $r(\cdot,\la)$, and $\varkappa(\cdot,\la)$ no longer need to be bounded near $\beta$.
\end{tabular} 
\\
Thirdly, the weaker assumptions required the use of new techniques to prove e.g.\ the relation between the essential spectra of the matrix operator 
\eqref{mo} and its Schur complement. Fourthly, the weaker assumptions enabled us to cope with the singularity at the boundary of the star in the 
symmetric stellar equilibrium model for which only a conjecture was made in \cite{ILLT13}.
Finally, we give the first comprehensive and systematic analysis for the singular part of the essential spectrum in terms 
of the asymptotic coefficients of the leading coefficient $\tpi(\cdot,\la)$ as $t\nearrow\beta$ of the Schur complement.

The paper is organized as follows. Section 2 contains the operator theoretic framework for the singular matrix differential operator \eqref{mo} and its Schur complement \eqref{sc}. 
Section 3 is dedicated to the characterization of the regular and singular part or the essential spectrum in terms of the Schur complement. 
Section 4 provides criteria for $\sesssing(\cA)$ to be empty and, if it is not empty, an explicit description in terms of certain limits of functions formed out
of the coefficients of the Schur complement. Section 5 contains some useful 
sufficient conditions for the assumptions in our main results and a more elegant formula for $\sesssing(\cA)$. Section 6 deals with the 
topological structure of $\sesssing(\cA)$. Section 7 shows that the problems considered in earlier works concern special examples of our 
general operators \eqref{mo} and special cases of our abstract classification in terms of $\tpi_0(\la)$, $\tpi_1(\la)$. In Section 8, we 
prove the conjecture that the singular part of the essential spectrum for the symmetric stellar equilibrium model is empty.
      

\section{Singular matrix differential operators and associated Schur complement} 
\label{sec:A.def}

In this section, we introduce the operator setting for matrix differential operators of the form  \eqref{mo} and the associated Schur complement \eqref{sc}, 
together with some basic assumptions. To this end, let $\alpha\in\RR$ and $\beta\in\RR\cup\{\infty\}$ with $\alpha<\beta$. On the interval $[\alpha,\beta)$, we introduce the scalar, vector, and matrix differential expressions
\begin{alignat*}{2}
  \tau_A &:= -\dt p \dt + q, & \qquad \tau_B &:= -\dt b^*  + c^*, \\[1ex]
  \tau_{B}^+ &:= b \dt + c, & \tau_{D} &:=D,
\end{alignat*}
with coefficient functions $p,q\!:\!\citv\!\to\!\RR$, $b\!=\!(b_i)^n_{i=1}$, $c\!=\!(c_i)^n_{i=1}\!:\!\citv\!\to\!\CC^n\!$, and $D\!=\!(d_{ij})_{i,j=1}^n\!:\!\citv\!\to\!\CC^{n\times n}$ satisfying the following assumptions; here $b^*$ and $c^*$ denote the (pointwise) row vector adjoints of $b$ and~$c$. 

\begin{ass}\label{ass:regularity}
Let $p\in C^1(\citv,\RR)$ with $p>0$, $q \in C(\citv,\RR)$, $b,c \in C^1(\citv,\CC^n)$ and $D\in C^1(\citv,\CC^{n\times n})$ with $D(t)^*=D(t)$ for $t\in\citv$.
\end{ass}
The differential expressions $\tau_A, \tau_B$, $\tau_B^+$, and $\tau_D$, respectively, induce operators
$A_0:L^2(\itv) \to L^2(\itv)$, $B_0:\bigl(L^2(\itv) \bigr)^n \to L^2(\itv)$, $C_0:L^2(\itv) \to \bigl(L^2(\itv)\bigr)^n$, and $D_0:\bigl(L^2(\itv)\bigr)^n \to \bigl(L^2(\itv)\bigr)^n$ 
on the domains
\begin{equation*}
\cD(A_0) := C_0^2(\itv),  \quad\cD(B_0) := \bigl(C_0^1(\itv)\bigr)^n,
\quad 
\cD(C_0) := C_0^1(\itv),  \quad\cD(D_0) := \bigl(C_0(\itv)\bigr)^n,
\end{equation*}
where $C_0^k(\itv)$, $k\in\N_0$, denotes the space of all functions $f\in C^k(\itv)$ with compact support in $\itv$.

In the Hilbert space $\cH:=L^2(\itv) \oplus \bigl(L^2(\itv)\bigr)^n$, we define the matrix differential operator
\begin{equation}\label{A02}
\begin{aligned} 
  \cA_0 & := \begin{pmatrix} A_0 & B_0 \\ C_0 & D_0 \end{pmatrix} =
   \left(\begin{array}{c|ccc} 
    -\dt p \dt + q & - \dt \ds\overline{b_1} + \overline{c_1} & \dots  & - \dt \overline{b_n} + \overline{c_n}\\[6pt]\hline\\[-8pt]
    b_1 \dt + c_1  & d_{11}                      & \dots  & d_{1n} \\
    \vdots         & \vdots                      & \ddots & \vdots \\
    b_n \dt + c_n  & d_{n1}                      & \dots  & d_{nn}  
   \end{array}\right)
\end{aligned}
\end{equation}
with domain $\cD(\cA_0) :=  \cD(A_0) \oplus \cD(B_0) = C_0^2(\itv) \oplus \bigl(C_0^1(\itv)\bigr)^n$. 

\begin{proposition}\label{prop:adj}
The operator matrix $\cA_0$ in \eqref{A02} is symmetric in $\cH$ with
\begin{equation}
\begin{aligned}
\label{A0adj}
  \cD(\cA_0^*) &\! =\! \biggl\{\binom{y_1}{y_2}\in \cH \colon
  y_1,\,py_1'+b^*y_2 \in {\rm AC_{loc}}(\citv), \\
   &\hspace{2.2cm} -\bigl(py_1'+b^*y_2\bigr)'\!+qy_1+c^*y_2 \in L^2(\itv), \ by_1'+cy_1+Dy_2 \in \bigl(L^2(\itv)\bigr)^n\biggr\}, \hspace*{-5mm}\nonumber \\
  \cA_0^*\binom{y_1}{y_2}
  &\!=\!\! 
  \begin{pmatrix} 
  -\bigl(py_1'+b^*y_2\bigr)'+qy_1+c^*y_2 \\[1.5ex]
  by_1'+cy_1+Dy_2 
  \end{pmatrix},
\end{aligned}
\end{equation}
and ${\rm dim\, ker}(\cA_0^*-\lambda)\le 2$ for $\la\in\mathbb C$.
\end{proposition}
\begin{proof}
The proof is similar to the proof of \cite[Proposition~2.3]{ILLT13} in the case $n=1$ where $D$ is a scalar function, and is thus left to the reader.
\end{proof}

For a densely defined closed linear operator $T$, we use the following definition of essential spectrum 
\[
  \sess (T) := \{ \la \in \CC : T-\la \text{ is not Fredholm}\},
\]
which is the set $\sigma_{e3}(T)$ in \cite[Sections~I.3 and IX.1]{EE87}. Note that all definitions of the essential spectrum in \cite[Sections~I.3 and IX.1]{EE87} are equivalent for self-adjoint operators.

Throughout this paper, $\cA$ denotes an arbitrary closed symmetric extension of $\cA_0$.
Since the deficiency indices of $\cA_0$ are finite, $\cA$ is a finite-dimensional extension of $\cA_0$ and hence (see e.g. [7, Section IX.4]) 
\[
  \sess(\cA) = \sess(\ov\cA_0)\subset\RR. 
\]
As in \cite{ILLT13}, we employ Glazman's decomposition principle to determine $\sess(\cA)$ (see \cite{Gla66}).
To this end, for an open subinterval $J\subset \citv$, we denote by $\cA_J$ the closure of the symmetric operator $\cA_{0,J}$ in $L^2(J)\, \oplus \, (L^2(J))^n$ generated by the restriction of $\cA_0$ to $C^2_0(J)\oplus \bigl(C^1_0(J)\bigr)^n$,~i.e.
\begin{equation}\label{clos1}
\cA_J=\ov{\cA_{0,J}}=\ov{\cA_0 \restriction {C^2_0(J)\oplus \bigl(C^1_0(J)\bigr)^n}}. 
\end{equation}
Then, for arbitrary  $t_0\in \itv$, the operator~$\cA$~in $L^2(\itv) \oplus (L^2(\itv))^n$ is a 
finite-dimensional extension of the orthogonal sum $\cA_{(\alpha,t_0)}\oplus\cA_{(t_0,\beta)}$ and hence
(see e.g.\ \cite[Section~IX.5.2]{EE87}) 
\begin{equation}
\label{glaz}
\sess(\cA)=\sess(\cA_{(\alpha,t_0)})\,\cup\,\sess(\cA_{(t_0,\beta)}).
\end{equation}

The first Schur complement $S_0(\la)$ of the operator matrix $\cA_0$ in \eqref{A02} 
which is defined for all $\la\in\CC\setminus \sigma(\overline{D_0})$ and acts in the first space component $L^2(\itv)$ (see \cite[Section 2.2]{Tre08})
is induced by the scalar second order differential expression
\begin{equation*}
\begin{aligned}
 \tau_S(\la) = \tau_A-\la-\tau_B(\tau_D-\la)^{-1}\tau_B^+ = -\dt p \dt + q - \la  - \left(-\dt b^*+c^*\right) (D-\la)^{-1} \left(b\dt+c\right).
\end{aligned}
\end{equation*}      
The differential expression $\tau_S(\la)$ can be rewritten in the standard symmetric form as 
\begin{equation}\label{Ssymm}
\tau_S(\la) = - \pt \pi(\cdot,\la)\pt + \ii\Bigl(r(\cdot,\la)\pt + \pt r(\cdot,\la)\Bigr) +   \tka (\cdot,\la)
\end{equation}
where, for $\la\in\CC\setminus \sigma(\overline{D_0})$,
\begin{equation}\label{pi.r.varkap.def}
\begin{aligned}
  \pi(\cdot,\la) &:= p-b^*(D-\la)^{-1}b, 
  \\[0.5mm] 
  r(\cdot,\la) &
  \defeq \Imag(b^*(D-\la)^{-1}c),
  \\[-0.8mm]
  \tka(\cdot,\la) &:= q - \la - c^*(D-\la)^{-1}c + \frac{\partial}{\partial t}\Real(b^*(D-\la)^{-1}c).
\end{aligned}
\end{equation}
Note that we use partial derivatives here and in the sequel since the coefficients now also depend on the spectral parameter $\la\in\CC$.

\smallskip

For an open subinterval $J\subset\citv$, let $D_{0,J}$ denote the multiplication operator by the matrix function $D$ in $\bigl(L^2(J)\bigr)^n$ with domain $\cD(D_{0,J}) = (C_0(J))^n$. Then, for $\la\in\RR\setminus\sigma(\ov{D_{0,J}})$, the formally symmetric differential expression $\tau_S(\la)$ induces a symmetric operator $S_{0,J}(\la)$ in $L^2(J)$ with domain $C^{2}_0(J)$; we denote the closure of $S_{0,J}(\la)$ by $S_J(\la)$, 
\begin{equation}
\label{clos2}
  \cD(S_{0,J}(\la)):=C^2_0(J), \quad S_{0,J}(\la)u:=\tau_S(\la)u, \qquad
  S_J(\la) := \ov{S_{0,J}(\la)}.
\end{equation}

\smallskip

To describe the resolvent sets of matrix multiplication operators such as $\overline{D_0}$, we introduce the following notation. For a subinterval $J\subset\citv$ and a matrix function $M: \citv \to \CC^{n\times n}$, we have
	\begin{equation}\label{Lam.J.def}
	\Lambda_J(M) 
	\defeq \{ \lambda \in \CC : \exists \, t_0 \in J, \ \det(M(t_0)-\lambda) = 0 \}
	=  \bigcup_{t\in J} \sa(M(t))\,;
	\end{equation}
notice that $\ov{\Lambda_J(M)}= \sigma(M)$ if $M$ is viewed as a matrix multiplication operator acting in $(L^2(J))^n$, see \cite{HW96}.
In the limiting case, we set
\begin{equation}\label{Lam.beta.def}
\Lambda_\beta(M) 
\defeq  \CC \setminus \{ \la \in \CC : \exists \, t_\la \!\in\! \itv \ \exists \, R_\la \!>\! 0 \ \forall \, t \!\in\! \itvla \  \la \!\notin\! \sigma(M(t)) \wedge
 \|(M(t) - \la)^{-1}\| \leq R_\la \};
\end{equation}
since $\Lambda_\beta(M)$ is the complement of the region of boundedness of the matrix family $(M(t))_{t\in[\alpha,\beta)}$, it is closed in $\CC$ 
(see \cite[Theorem~VIII.1.1]{Kato}). If the matrices $M(t)$, $t \in J$, are Hermitian, then the property
\begin{equation}
\|(M(t)-\lambda)^{-1}\| = \frac{1}{\dist(\lambda,\sigma(M(t)))}, \quad t\in J,
\end{equation}
implies that 
\[
\Lambda_\beta(M) = \{ \la \in \RR : \liminf_{t \nearrow \beta} \dist(\la, \sigma(M(t)))=0 \}.
\]

Throughout this paper, 
for the Hermitian matrix-valued function $D=(d_{ij})_{i,j=1}^n: \citv \to \CC^{n\times n}$ in the operator matrix \eqref{A02}, we denote by $\la_k(t) \in \RR$, $k=1,2,\ldots,n$, the eigenvalues of $D(t)$ for $t\in \citv$. Then
\[
\Lambda_J(D) = \bigcup_{t\in J} \, \big\{\la_k(t): k=1,2,\dots,n\big\} \subset \RR;
\]
in Section \ref{sec:str.of.sing.part}, we will assume that the possibly improper limits $\la_{k,\beta} \defeq \lim_{t\nearrow\beta}\la_k(t) \in \RR \cup \{-\infty,\infty\}$ exist, 
see Assumption (T1), in which case we will have
\begin{equation}\label{Lam.beta.lim}
\Lambda_\beta(D) = \{\la_{1,\beta},\dots,\la_{n,\beta}\} \cap \RR.
\end{equation}


\section{Essential spectrum and Schur complement}
\label{sec:3}

The essential spectrum of singular matrix differential operators consists of two parts, 
one due to the matrix structure which persists even when the operator is restricted to compact subintervals $[\alpha,\beta_r]\subset[\alpha,\beta)$ 
and one due to the singularity at $\beta$.
Our main tool to describe both the regular part and the singular part of the essential spectrum of $\cA$ is the first Schur complement $S(\la)$ introduced in Section \ref{sec:A.def}, and in particular its leading coefficient $\pi(\cdot,\la)$, see \eqref{Ssymm}.

\subsection{The regular part of the essential spectrum} 

For points $\la\in\RR\setminus\Lambda_{\citv}(D)$ for which $\pi(t,\la) \!=\! 0$ for some $t\!\in\! \citv$, the leading coefficient of the differential expression $\tau_S(\la)$ vanishes.
Below we show that these points give rise to essential spectrum of a restriction $\cA_{(\alpha,t_0)}$ to some finite interval $(\alpha,t_0)\subset (\alpha,\beta)$ 
with~$t_0<\beta$.

To see this, we first relate the zeros of $\pi(\cdot,\la)$ to the spectrum of the Hermitian matrix-valued function $\Delta\!:\! [\alpha,\beta)\!\to\! \CC^{n\times n}$ given~by
\begin{equation} 
\label{defDelta}
  \Delta(t) \defeq D(t)-\frac1{p(t)}b(t)b(t)^*, \quad t\in \citv, 
\end{equation}
which was used in \cite{ALMS94} to characterize the essential spectrum of operator matrices such as $\cA_{(\alpha,t_0)}$.

\begin{lemma}\label{lem:pidet}
For every $\la\in\RR\setminus\Lambda_{\citv}(D)$,  
\begin{align} \label{pidelta}
\pi(\mydot,\la)=p \frac{\det(\Delta-\la)}{\det(D-\la)} \quad \text{on} \quad \citv.
\end{align}
Moreover, for every open subinterval $J\subset \citv$ and every $\la\in\RR \setminus (\Lambda_J(D) \cup \Lambda_J(\Delta))$,
\begin{align}\label{identity2}
\frac{1}{\pi(\mydot,\la)}b^*\bigl(D-\la\bigr)^{-1}c = \frac{1}{p}b^*\bigl(\Delta-\la\bigr)^{-1}c \quad \text{on } J.
\end{align}
\end{lemma}
\begin{proof}
Sylvester's determinant theorem states that, for matrices $M_1\in\CC^{k\times m}$ and $M_2\in\CC^{m\times k}$,
\begin{align*}
\det(I_k-M_1M_2)=\det(I_m-M_2M_1);
\end{align*}
here $I_k \in \CC^{k\times k}$ and $I_m \in \CC^{m\times m}$ are the identity matrices. Applying this equality with $k=1$, $m=n$, $M_1=p^{-1}b^*$ and $M_2=(D-\la)^{-1}b$, we can rewrite $\pi(\cdot,\la)$ defined in \eqref{pi.r.varkap.def} as
\begin{align*}
  \pi(\mydot,\la) &= p\,\det\bigl(1-p^{-1}b^*(D-\la)^{-1}b\bigr) 
  = p\,\det\bigl(I_n-(D-\la)^{-1}bp^{-1}b^*\bigr) \\
  &= p\,\det\bigl((D-\la)^{-1}(D-\la-bp^{-1}b^*)\bigr) 
  = p\,\det\bigl((D-\la)^{-1}\bigr)\det\bigl(D-\la-bp^{-1}b^*\bigr) \\
  &= p\frac{\det(\Delta-\la)}{{\det(D-\la)}} \quad \text{on } J. 
\end{align*}

To prove the second claim, let $\la\!\in\!\RR \!\setminus\! (\Lambda_J(D) \cup \Lambda_J(\Delta))$. By the second resolvent identity and \eqref{defDelta}, we~have
\begin{equation*}
b^*(\Delta-\la)^{-1}c = b^*(D-\la)^{-1}c+\frac{1}{p}b^*(D-\la)^{-1}bb^*(\Delta-\la)^{-1}c \quad \text{on } J.
\end{equation*}
Hence, by the definition of $\pi(\cdot,\la)$ in \eqref{pi.r.varkap.def},
\begin{equation*}
\pi(\cdot,\la)\frac1p b^*(\Delta-\la)^{-1}c 
= b^*(\Delta-\la)^{-1}c-\frac{1}{p}b^*(D-\la)^{-1}bb^*(\Delta-\la)^{-1}c
= b^*(D-\la)^{-1}c\quad \text{on } J . 
\qedhere
\end{equation*}
\end{proof}
\begin{remark}\label{schurdef}
	If Assumption \ref{ass:regularity} is satisfied, then, for every open subinterval $J \subset \citv$ and every $\la\in \RR \setminus (\Lambda_J(D) \cup \Lambda_J(\Delta))$, the differential expression
	$\tau_S(\la)$ satisfies the conditions \cite[III.(10.3)]{EE87}, i.e.  
	\begin{enumerate}[label={\upshape(\roman{*})}]
	\item	$  \pi(\cdot,\la) \ne 0, \quad \dfrac 1{\pi(\cdot,\la)} \in L^1_{\rm loc}(J),$ \vspace{2mm}
	\item   $  \dfrac {2r(\cdot,\la)}{\pi(\cdot,\la)} = \dfrac{2\Im (b^*\,(\Delta-\la)^{-1} c)}{p} \in L^1_{\rm loc}(J), \quad \kappa(\cdot,\la) := \varkappa(\cdot,\la) + \ii \pt r(t,\la) \in L^1_{\rm loc}(J)$.
	\end{enumerate}
Here the identity in {\rm (ii)} follows from \eqref{identity2}. Since, in addition, $r(\cdot,\la) \in {\rm AC_{loc}}(J)$, the symmetric operator $S_{0,J}(\la)$ for $\la\in \RR \setminus (\Lambda_J(D) \cup \Lambda_J(\Delta))$ also satisfies the conditions of \cite[Theorem III.10.7]{EE87} and hence the deficiency numbers of $S_{0,J}(\la)$ are~$\le 2$.
\end{remark}

On compact intervals, the essential spectrum of top-dominant matrix differential operators was characterized in \cite{ALMS94}. 
As a consequence of this and Glazman's decomposition principle \eqref{glaz},
the closure $\ov{\Lambda_{J}(\Delta)}$ of the range of eigenvalues of the matrix function $\Delta$ on any open subinterval $J \subset \itv$ belongs to the essential spectrum of $\cA_J$.  This part of the essential spectrum of 
$\cA_J$  is called \textit{the regular part} and denoted by $\sessreg(\cA_J)$.
\begin{proposition} 
\label{prop:regpart}
Let Assumption \ref{ass:regularity} be satisfied and let $J$ be an open interval such that $\ov J \subset \citv$. If $\cA_J$ is the closed symmetric operator defined in \eqref{clos1}, then
\begin{equation*}
\sess(\cA_J) = \sessreg(\cA_J) = \ov{\Lambda_{J}(\Delta)} = \Lambda_{\ov J}(\Delta);
\end{equation*}
if $J=\itv$ and $\cA$ is any closed symmetric extension of the operator $\cA_0$ defined in \eqref{A02}, then
\begin{equation*}
\sess(\cA) \supset \sessreg(\cA) = \ov{\Lambda_{\citv}(\Delta)};
\end{equation*}
in particular, $\ov{\Lambda_{\citv}(\Delta) }=\RR$ implies that $\sess(\cA)=\RR$.
\end{proposition}
\begin{proof}
The reasoning is completely analogous to the proof of \cite[Proposition 3.3]{ILLT13} in the case $n=1$; it uses Glazman's decomposition principle \eqref{glaz} and \cite[Theorem 4.5]{ALMS94}.
\end{proof}

\subsection{The singular part of the essential spectrum}

While for intervals $J$ such that $\ov J \subset \citv$ the essential spectrum of $\cA_J$ is exhausted by its regular part, the singular endpoint $\beta$ may give rise to an additional part of $\sess(\cA)$. This part of the essential spectrum of $\cA$ is referred to as \emph{the singular part} and denoted by $\sesssing(\cA)$ (not to be confused with the singular continuous spectrum).

In the following, we characterize the essential spectrum of the restrictions $\cA_{(t_0,\beta)}$ for suitably large $t_0\in(\alpha,\beta)$ in terms of the Schur complement. The following results generalize those in \cite{ILLT13} not only to the more general case $n\geq 1$. More importantly, we develop a different proof that allows us to weaken the assumptions on the coefficient functions in \eqref{pi.r.varkap.def} considerably (cf.\ \cite[Assumption~(B)]{ILLT13}).

The new Assumption \ref{ass:Schur} below contains a weight function $\eta$, which enables us to cover a larger set of operators; particular weights have already been used in the method of \cite{KLN08}. The choice of a suitable $\eta$ depends on the behaviour of other coefficients in the Schur complement (see the proof of Theorem \ref{prop:suff} for examples).
\begin{lemma}\label{lem:tlam.1}
Let Assumption \ref{ass:regularity} be satisfied. 
Then, for every $\la \in \RR \setminus ( \sessreg(\cA) \cup \Lambda_\beta(D))$, there exists $\widehat t_\la \in \itv$ such that 
\begin{enumerate}[label=\rm{(\roman{*})}]
\item\label{lam.D} $\lambda \notin \Lambda_{\citvlatil}(D)$ and $\sup_{t \in \citvlatil}\|(D(t)-\lambda)^{-1}\| < \infty$;
\item $\pi(\mydot,\lambda) \neq 0$ on $\citvlatil$, hence either $\pi(\mydot,\lambda) >0$ or $\pi(\mydot,\lambda) <0$ on $\citvlatil$.
\end{enumerate}
\end{lemma}
\begin{proof}
The existence of $\widehat t_\la$ for which \ref{lam.D} holds is immediate from the definition of $\Lambda_{\beta}(D)$, see \eqref{Lam.beta.def}. The function $\pi(\cdot,\lambda)$ is non-zero on $\citvlatil$ since $p>0$ on $\citv$, $\la \in \RR \setminus ( \ov{\Lambda_{\citv}(\Delta)} \cup \Lambda_\beta(D))$ and the relation \eqref{pidelta} holds. Because $\pi(\cdot,\lambda)$ is continuous on $\citvlatil$, it cannot change sign on $\citvlatil$. 
\end{proof}

\begin{ass}\label{ass:Schur}
For every $\la\!\in\! \RR \!\setminus\! \big( \sessreg(\cA) \cup \Lambda_\beta(D) \big)$, there is a $t_\la \!\in\! \citvlatil$ with $\widehat t_\la$ as in
Lemma~\ref{lem:tlam.1} such~that
\begin{enumerate}[label={{\upshape(B\arabic{*})}}]
\item \label{ass:Schur.1}
there exists a constant $k_1 >0$ with
\begin{equation}\label{pi.B1a}
|\pi(\mydot,\la)| >  k_1\bigl\|(D-\la)^{-1}b\bigr\|^2_{\CC^n} \quad \text{on } \citvla;
\end{equation}
\item \label{ass:Schur.2}
there exists a constant $k_2 >0$ and a positive-valued function $\eta\in C^2(\citvla)$ with
\begin{equation}\label{Veta.below}
\hspace*{7mm}
V_{\eta}(\cdot,\la)  -k_2\bigg\|(D-\la)^{-1}\Bigl[\Bigl(\ii\frac{r(\cdot,\la)}{\tpi(\cdot,\la)}+\frac{\eta'}{2\eta}\Bigr)b+c\Bigr]\bigg\|^2_{\CC^n} \text{bounded from below on } \citvla
\end{equation}
where 
\begin{equation}
\label{V_eta}
V_{\eta}(\cdot,\la):=
s_\pi \left(
\tka(\cdot,\la)-\frac{r(\cdot,\la)^2}{\tpi(\cdot,\la)} - \frac{1}{2\sqrt{\eta}}\Bigl(\tpi(\cdot,\la)\frac{\eta'}{\sqrt{\eta}}\Bigr)'
\right)
\end{equation}
and 
\begin{equation}\label{spi.def}
s_{\pi}:= \sgn(\pi(\cdot,\la)\!\!\restriction\!_{\citvla}) =
\begin{cases}
\hspace{3mm} 1 & \text{if } \pi(\cdot,\la) > 0 \text{ on } \citvla, 
\\
 -1 & \text{if } \pi(\cdot,\la) < 0 \text{ on } \citvla. 
\end{cases}
\end{equation}
\end{enumerate}
\end{ass}

\begin{lemma} \label{lem:AS1}
Let Assumptions \ref{ass:regularity}, \ref{ass:Schur} be satisfied and suppose $\sessreg(\cA) \ne \RR$. Let $\la \in \RR \setminus (\sessreg(\cA) \cup \Lambda_\beta(D))$ and let $t_\la \in \citv$ be as 
in Assumption \ref{ass:Schur}. Further, let $S_{\itvla}(\la)$ be defined as in \eqref{clos2} with $J \defeq \itvla$. 
Then 
\begin{align}
\label{Aequivschur}
  \lambda\in\sess(\cA_{\itvla }) \ &\iff \
  0\in\sess(S_{\itvla}(\la)); \\
\intertext{moreover,} 
\label{A.ess.dec}
\lambda\in\sess(\cA) \ &\iff \
0\in\sess(S_{\itvla}(\la)).
\end{align} 
\end{lemma}
\begin{proof}
	We fix 
	$
	\la \in \RR \setminus (\sessreg(\cA) \cup \Lambda_\beta(D))
	$
	and abbreviate $J \defeq \itvla$, $\mathcal{H}_1 := L^2(J)$, and $\mathcal{H}_2 \!:=\! \bigl(L^2(J)\bigr)^n$. 
	We equip $\mathcal{H}_2=\mathcal{H}_1^n$ with the norm
	\[
	\|f\|_{\mathcal{H}_2} := \Bigl(\|f_1\|^2_{\mathcal{H}_1}+\|f_2\|^2_{\mathcal{H}_1}+\ldots+\|f_n\|^2_{\mathcal{H}_1}\Bigr)^{\frac{1}{2}}, 
	\quad f=(f_1,f_2,\ldots,f_n)^{\mathrm{t}} \in \mathcal{H}_2 = \mathcal{H}_1^n.
	\]
	By definitions \eqref{clos1} and \eqref{clos2}, we have $\cA_{J}=\ov{\cA_{0,J}}$ and $S_J(\la) = \ov{S_{0,J}(\lambda)}$. Thus we have to prove that 
	\begin{equation} 
	\label{hot}
	\lambda\in\sess(\ov{\cA_{0,J}}) \quad\Longleftrightarrow\quad
	0\in\sess\bigl(\ov{S_{0,J}(\lambda)}\bigr).
	\end{equation}
	Note that, by Lemma \ref{lem:tlam.1}, $D(t)-\lambda$ is invertible for $t\!\in\citvla$ and 
	$\sup_{t \in \citvla}\|(D(t)-\lambda)^{-1}\| \!<\! \infty$. Moreover, since $D\in C^1([\alpha,\beta),\CC^{n\times n})$ by Assumption \ref{ass:regularity}, 
	we have $(D-\la)^{-1} C^1_0([\alpha,\beta),\CC^n) \subset C^1_0([\alpha,\beta),\CC^n)$.

	``$\Longleftarrow$" in \eqref{hot}: The proof of this implication is completely analogous to the  corresponding part of the proof of \cite[Lemma 3.7]{ILLT13}.
	
	``$\Longrightarrow$" in \eqref{hot}: In order to delete the first order derivative in $S_{0,J}$, we apply
	the unitary transformation 
	\begin{equation}\label{omega}
	U:L^2({J}) \to L^2(J), \quad (Uf)(t)=\e^{-\ii\omega(t,\la)}f(t), 
	\quad \text{with} \quad
	\omega(t,\la) :=\int^t_{t_\la}\frac{r(s,\la)}{\pi(s,\la)}\text{d}s;
	\end{equation}
	note that $\frac{r(\cdot,\la)}{\tpi(\cdot,\la)}\in C^{1}(J)$.
	Then the transformed operator is again symmetric and has the form
	\begin{align*}
	& T_0(\la) 
	= 
	US_{0,J}(\la)U^{-1} 
	= 
	-\pt\tpi(\cdot,\la)\pt + \tka(\cdot,\la)-\frac{r(\cdot,\la)^2}{\tpi(\cdot,\la)},
	\qquad 
	\mathcal{D}(T_0(\la)) = U\mathcal{D}(S_{0,J}(\la)).
		\end{align*}
	It is not difficult to check that, with $V_\eta$ defined as in \eqref{V_eta},
	\begin{equation}\label{T_0}
	T_0(\la) = s_\pi \left(
	-\frac{1}{\sqrt{\eta}}\pt \eta|\pi(\cdot,\la)| \pt\frac{1}{\sqrt{\eta}} + V_{\eta}(\cdot,\la)
	\right).
	\end{equation}
	Since $V_\eta$ is bounded from below by \eqref{Veta.below} in Assumption \ref{ass:Schur}, there is a $\delta\geq 0$ such that $V_{\eta}(\cdot,\la)+\delta\geq 0$ on $\citvla$ and, for $f\in\mathcal{D}(T_0(\la))$, 
	\begin{equation*}
	\begin{aligned}
	\big( (T_0(\la) + 2 s_\pi \delta)f,f \big)_{\mathcal{H}_1} &= s_\pi 
	\left(
	-\Big( \pt \Big( \eta |\tpi(\cdot,\la)| \pt \Big( \frac{1}{\sqrt{\eta}} f \Big) \Big), \frac{1}{\sqrt{\eta}} f \Big)_{\mathcal{H}_1} + \big( V_{\eta} f, f \big)_{\mathcal{H}_1} + 2 \delta \|f\|_{\mathcal{H}_1}^2 
	\right)
	\\
	& = s_\pi \left( 
	\Big\|\sqrt{\eta|\tpi(\cdot,\la)|}\pt  \Big( \frac{1}{\sqrt{\eta}} f \Big) \Big\|_{\mathcal{H}_1}^2 + \|\sqrt{V_{\eta}+\delta}f\|_{\mathcal{H}_1}^2 + \delta\|f\|_{\mathcal{H}_1}^2
	\right).
	\end{aligned}
	\end{equation*}
	Thus $T_0(\la) +2 s_\pi \delta$ is uniformly positive if $s_\pi=1$ and uniformly negative if $s_\pi\!=\!-1$ with ${s_\pi \!=\! \sgn (\pi(\cdot,\la)\!\!\restriction_{\citvla})}$. Hence the quadratic form generated by $T_0(\la)$ is closable. We denote its closure by $t_F(\la)$
and by $T_F(\la)$ the self-adjoint operator associated with $t_F$, i.e.\ the Friedrichs extension of $T_0(\la)$ (cf.\ \cite[VI.\S~2]{Kato}).
Note that the domain of $t_F(\la)$ is the closure of $\mathcal{D}(T_0(\la))$ with respect to the norm
	\[
	\Bigl( \Big\|\sqrt{\eta|\tpi(\cdot,\la)|}\pt \Big(  \frac{1}{\sqrt{\eta}} \cdot \Big) \Big\|_{\mathcal{H}_1}^2+\|\sqrt{V_{\eta}+\delta}\cdot\|_{\mathcal{H}_1}^2+\|\cdot\|_{\mathcal{H}_1}^2 \Bigr)^{\frac{1}{2}}.
	\]
	
Now we suppose that $0\notin\sess\bigl(\ov{S_{0,J}(\lambda)}\bigr)$. Since $U$ is unitary and $S_{0,J}$ has finite deficiency indices, we have $0\notin\sess\bigl(T_F(\lambda)\bigr)$. Let $P_0$ be the spectral projection onto the eigenspace of $T_F(\la)$ corresponding to $0$ which is $\{0\}$ if $0$ is not an eigenvalue of $T_F(\la)$. Since $0$ is not an eigenvalue of infinite multiplicity, $P_0$ is of finite rank and thus compact. Hence $0\notin\sess\bigl(T_F(\lambda)-P_0\bigr)$;
note that $0\notin\sigma_{\text{p}}\bigl(T_F(\lambda)-P_0\bigr)$ by definition of $P_0$. Then
 the operator
	\[
	K=\begin{pmatrix}
	K_0 & 0_{1\times n} \\[2.5ex]
	0_{n\times 1} & 0_{n\times n}
	\end{pmatrix}, \quad K_0:=U^{-1}P_0U,
	\]
	is compact. It is not difficult to check that $\la\in\sigma_{\text{p}}(\cA_{0,J}-K)$ implies that
	$0\in\sigma_{\text{p}}\bigl(S_{0,J}(\lambda)-K_0\bigr)\subset\sigma_{\text{p}}\bigl(T_F(\lambda)-P_0\bigr)$, a contradiction to the choice of $P_0$. Hence $\la\notin\sigma_{\text{p}}(\cA_{0,J}-K)$. 
	
	The claim is proved if we show that $(\cA_{0,J}-K-\la)^{-1}$ is bounded on $\ran(\cA_{0,J}-K-\la)$. In fact, we will show that the latter implies that $\la\notin\sigma_{\text{p}}(\cA_J-K)$, and thus,  by \cite[Lemma 2.4]{ILLT13}, $\la\notin\sess(\cA_J-K) = \sess(\cA_J)$.
	
	To see why the boundedness of $(\cA_{0,J}-K-\la)^{-1}$ on $\ran(\cA_{0,J}-K-\la)$ implies $\la\notin\sigma_{\text{p}}(\cA_{J}-K)$, 
suppose to the contrary that $\cA_{J}-K-\la$ is not injective. Then there exists an $x \in \cD(\cA_{J})$,  $\|x\|_{\cH_1\oplus \cH_2}=1$, such that $(\cA_{J}-K-\la)x=0$. Since $\cA_J=\ov{\cA_{0,J}}$, there exists a sequence $(x_n)_{n\in\N} \subset \cD(\cA_{0,J})$ with $x_n \to x$ and $(\cA_{0,J}-K-\la)x_n\to 0$, $n \to \infty$. Letting $y_n:=(\cA_{0,J}-K-\la)x_n\in \text{ran} (\cA_{0,J}-K-\la)$, we obtain
	\[
	\frac{\|(\cA_{0,J}-K-\la)^{-1}y_n\|_{\cH_1\oplus \cH_2}}{\|y_n\|_{\cH_1\oplus \cH_2}} 
	= \frac{\|x_n\|_{\cH_1\oplus \cH_2}}{\|(\cA_{0,J}-K-\la)x_n\|_{\cH_1\oplus \cH_2}} \to \infty,  \quad n \to \infty,
	\]
a contradiction to the boundedness of $(\cA_{0,J}-K-\la)^{-1}$ on $\ran(\cA_{0,J}-K-\la)$.
	
	It remains to be shown that $(\cA_{0,J}-K-\la)^{-1}$ is bounded on $\ran(\cA_{0,J}-K-\la)$.
	To this end, let $(f,g)^{\rm t} \!\in\! \ran(\cA_{0,J}-K-\la)$,
	\begin{align}
	\label{nov18}
	(A_0-K_0-\lambda)u+B_0 v &= f, \\
	\label{nov18-ii}
	C_0 u+(D_0-\lambda)v &= g,
	\end{align}
	with $(u,v)^{\rm t} \in \cD(\cA_{0,J}) = C^2_0(J) \oplus \bigl(C^1_0(J)\bigr)^n$. Then Assumption \ref{ass:regularity} implies that 
	$(D_0-\la)^{-1}C_0u \in \bigl(C^1_0(J)\bigr)^n$, and \eqref{nov18-ii} shows that $(D_0-\la)^{-1}g=(D_0-\la)^{-1}C_0u+v\in \bigl(C^1_0(J)\bigr)^n$.
	Solving \eqref{nov18-ii} for $v$, we can thus substitute $v=-(D_0-\lambda)^{-1}C_0u+(D_0-\lambda)^{-1}g \in \bigl(C^1_0(J)\bigr)^n$ into  \eqref{nov18} to obtain
	\[
	(A_0-K_0-\lambda)u-B_0(D_0-\lambda)^{-1}C_0u=f-B_0(D_0-\lambda)^{-1}g.
	\]
	Since the left-hand side equals $\bigl(S_{0,J}(\la)-K_0\bigr)u = U^{-1}\bigl(T_F(\la)-P_0\bigr)Uu$ and $U^{-1}\bigl(T_F(\la)-P_0)U$ is boundedly invertible, it follows that
	\[
	u=U^{-1}\bigl(T_F(\la)-P_0\bigr)^{-1}Uf-U^{-1}\bigl(T_F(\la)-P_0\bigr)^{-1}UB_0(D_0-\lambda)^{-1}g.
	\]
	Inserting this back into the above formula for $v$, we find that
	\[
	v=-(D_0-\lambda)^{-1}C_0U^{-1}\bigl(T_F(\la)-P_0\bigr)^{-1}Uf+(D_0-\lambda)^{-1}g
	+(D_0-\lambda)^{-1}C_0U^{-1}\bigl(T_F(\la)-P_0\bigr)^{-1}U\!B_0(D_0-\lambda)^{-1}g.
	\]
	
Now define the auxiliary operator 
	\begin{align*}
	L \defeq& \,l_1(\cdot,\la) \sqrt{\eta} \,\pt  \frac{1}{\sqrt{\eta}} +l_2(\cdot,\la),\\
	\mathcal{D}(L) \defeq& \,\Bigl\{f\in L^2(J):\,\Bigl( \frac{1}{\sqrt{\eta}}  f \Bigr)'\!\in\! L^1_{\text{loc}}(J), \ \sqrt{\eta\tpi}\Bigl(\frac{1}{\sqrt{\eta}}f\Bigr)'\!\in
	\! L^2(J), \ \sqrt{V_{\eta}+\delta}f\in L^2(J)\Big\} 
	\end{align*} 
	where 
	\begin{align*}
	l_1(\cdot,\la) = \e^{\ii\omega(\cdot,\la)}(D-\la)^{-1}b, \quad l_2(\cdot,\la) = \e^{\ii\omega(\cdot,\la)}(D-\la)^{-1}\Bigl[\Bigl(\ii\,\omega'(\cdot,\la)+\frac{\eta'}{2\eta}\Bigr)b+c\Bigr],
	\end{align*}
	and $\omega(\cdot,\la)$ is defined as in \eqref{omega}. Further, we introduce the operators
	\begin{align*} 
	F(\lambda) \defeq L\, |T_F(\la) + 2 s_\pi \delta|^{-\frac 12}, \quad G(\lambda) &:=|T_F(\la)+2 s_\pi \delta|^{\frac{1}{2}} \bigl(T_F(\la)+P_0\bigr)^{-1} 
	|T_F(\la) + 2 s_\pi \delta|^{\frac{1}{2}}.
	\end{align*}
Then $F(\la)$ is an extension of $(D_0-\la)^{-1}C_0U^{-1}|T_{F}(\la) + 2s_\pi \delta|^{-\frac 12}$ since $L$ is an extension of $(D_0-\la)^{-1}C_0U^{-1}$. Moreover, it can be verified directly from the definition of the adjoint that  
$L^* \supset UB_0(D_0-\la)^{-1}$. Thus 
	\[
	F(\la)^* \supset |T_{F}(\la)+2 s_\pi \delta|^{-\frac 12}U B_0 (D_0-\lambda)^{-1}.  
	\] 
Hence the above relations for $u$ and $v$ show that $\binom uv = (\cA_{0,J}-K-\la)^{-1} \binom fg$ is given by
	\[
	\begin{pmatrix}u\\[4mm] v   \!\end{pmatrix}\!\!=\!\!
	\begin{pmatrix}\!
	U^{-1}\bigl(T_F(\la)\!-\!P_0\bigr)^{-1}U
	\!\!&\!\! - U^{-1}\bigl(T_F(\la)\!-\!P_0\bigr)^{-1} |T_F(\la)\!+\!2 s_\pi \delta|^{\frac{1}{2}}  \,F(\lambda)^* \\[2ex]
	-F(\lambda)|T_F(\la) \!+\! 2 s_\pi \delta|^{\frac{1}{2}} \bigl(T_F(\la)\!-\!P_0\bigr)^{-1}U  
	\!\!&\!\! \:(D_0\!-\!\lambda)^{-1}\!+\!F(\lambda)  G(\lambda) \,F(\lambda)^*
	\!\end{pmatrix}
	\!\!\!\begin{pmatrix}\!f\\[4mm]  g \end{pmatrix}\!.
	\]

The operator $\bigl(T_F(\la)-P_0\bigr)^{-1}|T_F(\la) + 2 s_\pi \delta|^{\frac{1}{2}}$ has a bounded extension to $\mathcal{H}_1$. If we show that $G(\la)$ has a bounded extension on $\mathcal{H}_1$ and $F(\la)$ is bounded on $\mathcal{H}_2$, then all entries in the above operator matrix are bounded and hence $(\cA_{0,J}-K-\lambda)^{-1}$ is bounded on $\ran(\cA_{0,J}-K-\la)$.

Since $G(\la)$ can be written as an orthogonal sum of bounded operators  
	\[
	G(\la)= s_\pi \left[ 
	\Bigl(I+2\delta (I-P_0) \bigl((I-P_0)T_F(\la)(I-P_0) \bigr)^{-1}\Bigr) \oplus 2\delta P_0
	\right],
	\]
the existence of a bounded extension of $G(\la)$ follows.
	
Next we prove the boundedness of $F(\la)$. It can be shown that $\mathcal{D}(t_F(\la))$ is contained in $\mathcal{D}(L)$. Thus $F(\la)$ is everywhere defined. 
For arbitrary $f\in \mathcal{H}_1$ 
	and $g:=|T_F(\la) + 2 s_\pi \delta|^{-1/2}f \in \mathcal{D}(t_F(\la))$, we obtain
	\begin{align*}
	\frac{\|F(\la)f\|_{\mathcal{H}_2}^2}{\|f\|_{\mathcal{H}_1}^2} = \frac{\|Lg\|_{\mathcal{H}_2}^2}{\||T_F(\la) + 2 s_\pi \delta|^{1/2}g\|_{\mathcal{H}_1}^2}.
	\end{align*}
	Hence the second representation theorem \cite[Theorem~VI.2.23]{Kato} yields
	\begin{align}
	\label{lastestimate}
	\frac{\|F(\la)f\|_{\mathcal{H}_2}^2}{\|f\|_{\mathcal{H}_1}^2} 
	= 
	\frac{\|Lg\|_{\mathcal{H}_2}^2}{s_\pi (t_F[g]+2 s_\pi \delta\|g\|_{\mathcal{H}_1}^2 ) } 
	\leq \frac{2\bigl\|\|\eta^{\frac{1}{2}}l_1(\cdot,\la)\|_{\CC^n}\bigl(\eta^{-\frac{1}{2}}g\bigr)'\bigr\|_{\mathcal{H}_1}^2 + 2\bigl\|\|l_2(\cdot,\la)\|_{\CC^n}g\bigr\|_{\mathcal{H}_1}^2}{\|\sqrt{\eta |\tpi(\cdot,\la)| }\partial_t \eta^{-\frac{1}{2}}g\|_{\mathcal{H}_1}^2 + \|\sqrt{V_{\eta}+\delta}g\|_{\mathcal{H}_1}^2 + \delta\|g\|_{\mathcal{H}_1}^2 }.                                                                                            %
	\end{align}
	Now \eqref{Veta.below} in  Assumption \ref{ass:Schur} implies the boundedness of $F(\la)$ on $\mathcal{H}_1$ 
	which completes the proof of \eqref{hot} and hence of \eqref{Aequivschur}.
	
	In order to prove \eqref{A.ess.dec}, we use that by Glazman's decomposition principle \eqref{glaz} 
	\begin{equation*}
	\la \in \sess(\cA) \iff  \la \in \big( \sess(\cA_{(\alpha,t_\la)})\,\cup\,\sess(\cA_{(t_\la,\beta)}) \big).
	\end{equation*}
	By Proposition \ref{prop:regpart}, $\sess(\cA_{(\alpha,t_\la)}) = \ov{ \Lambda_{[\alpha,t_\la)}} \subset \ov{ \Lambda_{\citv}} =  \sessreg(\cA)$. Since $\la \notin \sessreg(\cA)$ by assumption, it follows that $\la \notin \sess(\cA_{(\alpha,t_\la)})$. Now \eqref{A.ess.dec} follows from \eqref{Aequivschur}. 
\end{proof}
\section{Singular part of essential spectrum}
\label{sec:mainresults}

In this section, we analyse the singular part $\sesssing(\cA) = \sess(\cA) \setminus \sessreg(\cA)$ of the essential spectrum. 
Together with Proposition \ref{prop:regpart} describing the regular part of the essential spectrum, we thus 
obtain a full characterization of $\sess(\cA)$ up to the exceptional set $\Lambda_\beta(D)$.

First we provide conditions for $\sesssing(\cA) \setminus \Lambda_\beta(D)=\emptyset$. If $\sesssing(\cA) \setminus \Lambda_\beta(D)\ne \emptyset$, we establish an analytic description of this set in terms of the coefficients of the given operator matrix $\cA_0$ in \eqref{A02}.

\subsection{Criteria for $\sesssing(\cA) \setminus \Lambda_\beta(D) = \emptyset$}

By Lemma \ref{lem:AS1}, we know that $\sesssing(\cA) \setminus \Lambda_\beta(D) \!=\! \emptyset$ if $\sess(S_{\itvla}(\la)) \!=\! \emptyset$ for all $\la\in \RR \setminus \big( \sessreg(\cA) \cup \Lambda_\beta(D) \big)$. The latter holds for instance if $S_{\itvla}(\la)$ is in limit-circle case at $\beta$, see \cite[Theorem 10.12.1(2)]{Zettl}. The possibility of employing the limit-point/circle classification was mentioned in \cite{KLN08} and used (for a Hamiltonian system) in \cite{QCH11}.
While this provides only an implicit characterization, our conditions on the coefficient functions are explicit and not restricted to the limit-circle case at~$\beta$. 
Moreover, we refute the suspicion raised in \cite[p.\ 137]{KLN08} that limit-point case is crucial for $\sesssing(\cA)\ne \emptyset$ (see Example~\ref{conj-lc}~below).

\begin{ass}\label{ass:abs}
	Suppose that, for every $\la\in \RR \setminus \big( \sessreg(\cA) \cup \Lambda_\beta(D) \big)$, there exists $ \wt t_\la \!\in\! \citvla$ with $t_\la$ as in Assumption~\ref{ass:Schur} such~that one of the following holds.
\vspace{1mm}	
	\begin{enumerate}[label={{\upshape(C\arabic{*})}}]
	\item \label{ass:C1} $\eta \, h(\cdot,\la)^2 \in L^1( \itvlat)$ with
		\begin{equation}\label{h.def}
		h(t,\la) := 
		\begin{cases}
		\ds \left( \int_t^{\beta}\frac{\d s}{\eta(s) |\tpi(s,\la)| } \right)^{\frac 12} 
		& \text{if} \quad \dfrac 1{\eta \tpi(\cdot,\la)} \in L^1(\itvlat),
		\\[3ex]
		\ds \left( \int_{\wt t_\la}^t\frac{\d s}{\eta(s)|\tpi(s,\la)|} \right)^{\frac 12}
		& \text{otherwise};
		\end{cases}
		\end{equation}
\vspace{1mm}
	\item \label{ass:molchanov}
	\begin{enumerate}
	\item \label{molchanov01} $\ds\tpi(\mydot,\la)$, $\ds\frac{1}{\tpi(\mydot,\la)}$ are bounded on $\itvlat$;
	\item \label{molchanov02}
		$\ds s_{\tpi}\bigg(\varkappa(\cdot,\la) - \frac{r(\cdot,\la)^2}{\tpi(\cdot,\la)}\bigg)$ is bounded from below on $\itvlat$;

\vspace{1mm}
	moreover, if $\beta=\infty$, for all $d>0$,
		\begin{equation}\label{molchanov03}
		\ds\lim_{c\to\infty}\int_{c-d}^{c+d} s_{\tpi}\bigg(\varkappa(t,\la) - \frac{r(t,\la)^2}{\tpi(t,\la)}\bigg) \, \d t = \infty.
		\end{equation}
	\end{enumerate}
	\end{enumerate}
\end{ass}

The next theorem shows that if one of the conditions \ref{ass:C1} or \ref{ass:molchanov} in Assumption \ref{ass:abs} above is satisfied, the singular part of the essential spectrum of $\cA$ outside of $\Lambda_\beta(D)$ is empty.

\begin{theorem} \label{thm:ess.abs}
	Let Assumptions \ref{ass:regularity}, \ref{ass:Schur}, and \ref{ass:abs} be satisfied. Then, for every closed symmetric extension $\cA$ of $\cA_0$ defined in \eqref{A02}, $\sesssing(\cA) \setminus \Lambda_\beta(D) = \emptyset$, i.e.
	\begin{equation}
	\label{lastitem}
	\sess(\cA) \setminus \Lambda_\beta(D)
	= \sessreg(\cA) \, \setminus \Lambda_\beta(D).
	\end{equation}
\end{theorem}
\begin{proof}
	Let $\wt t_\la \in[t_\la,\beta)$ be as in Assumption \ref{ass:abs}. 
	The claim in \eqref{lastitem} follows from Lemma \ref{lem:AS1} if we verify that $\sess(S_{ \itvlat }(\la))=\emptyset$. Let $T_F(\la)$ be the Friedrichs extension of $T_0(\la)$ as in the proof of Lemma~\ref{lem:AS1}. Since $T_F(\la)$ is a finite dimensional extension of $US_{(\wt t_\la,\beta)}(\la)U^{-1}$ and $U$ is unitary, we have $\sess(T_F(\la))=\sess(S_{(\wt t_\la,\beta)}(\la))$. 
	
	First consider the case when Assumption \ref{ass:C1} holds. Since the operator 
	\begin{equation*}
	U_{\eta}:L^2(\itvlat)  \to L^2(\itvlat, \eta), 
	\quad
	U_{\eta}f:=\frac 1{\sqrt{\eta}} f
	\end{equation*}
	is unitary, we have $\sess(T_F(\la)) = \sess(U_{\eta}T_F(\la)U_{\eta}^{-1})$. Hence it suffices to show that 
	\begin{align}
	\label{essrol}
	\sess(U_{\eta}T_F(\la)U_{\eta}^{-1})=\emptyset.
	\end{align}
	With $\delta$ as in the proof of the Lemma~\ref{lem:AS1}, it is not hard to see that the operator
	\begin{align}
	\label{rol1}
	U_{\eta}T_F(\la)U_{\eta}^{-1} +s_\pi \delta= \frac{s_\pi}{\eta}\Bigl(-\pt\eta|\tpi(\cdot,\la)|\pt + \eta \bigl(V_{\eta}+\delta\bigr)\Bigr),
	\end{align}
	acting in the weighted Hilbert space $ L^2( \itvlat, \eta)$, coincides with the operator $T$ (up to the inessential overall sign $s_\pi \in \{-1,1\}$) in \cite{Rollins} with $m=\eta$, $p_0=\eta\bigl(V_{\eta}+\delta\bigr)$, and $p_1=\eta|\tpi(\cdot,\la)|$. By Assumptions \ref{ass:Schur} and \ref{ass:abs}, the differential operator in \eqref{rol1} satisfies all conditions of Theorem in \cite{Rollins}, and therefore, has compact resolvent, i.e.~\eqref{essrol} holds (cf.\ \cite[Theorem IX.3.1]{EE87}).
	
	Now consider the case when Assumption \ref{ass:molchanov} holds.  If $\beta< \infty$, then the form domain $\cD(t_F(\la))$ of $T_F(\la)$ with $\eta=1$ is a subset of $W^{1,2}(\itvlat)$ which is compactly embedded in $L^2(\itvlat)$ (cf.\ \cite[Theorem~6.3]{Adams-2003}). Hence $\sess(T_F(\la)) = \emptyset$ by \cite[Theorem~XIII.64]{Reed4}.
	
	If $\beta=\infty$, then Assumption \ref{ass:molchanov} guarantees that the assumptions of Molcanov's criterion are satisfied (cf. \cite[Theorem VIII.4.2]{EE87}) and the latter yields $\sess(T_F(\la)) = \emptyset$.
\end{proof}
\begin{example}
\label{conj-lc}
One of the main results of \cite{KLN08} implies that, for the operator $\cA_0$ considered in \eqref{KLNO8A_0}, see Example A  below, $\sesssing(\cA)\ne\emptyset$ if the Schur complement is in limit-point case at $\beta=0$, see \cite[Theorem~6.1]{KLN08}. The suspicion raised in \cite[p.\ 137]{KLN08} that this may be true in general is disproved by the following simple example.

Consider \eqref{A02} on $L^2((1,\infty))$ with $n=1$ and coefficient functions $p\equiv c\equiv 1$, $b\equiv 0$, $q(t)=t^2$, $D(t)=\frac{1}{t}$, $t\in[1,\infty)$. Then $\Delta(t)=\frac{1}{t}$, $t\in[1,\infty)$, $\Lambda_{\beta}(D)=\{0\}$ and the Schur complement is given by 
\[
S(\la) = -\frac{\text{d}^2}{\text{d}t^2} + V(t,\la), \quad V(t,\la) := t^2-\la-\frac{t}{1-\la t}, \quad \la\notin [0,1]. 
\]
It is easy to see that Assumptions \ref{ass:regularity}, \ref{ass:Schur}, and \ref{ass:molchanov} are satisfied. In particular, there is a $t_\la > 1$ such that $V(\cdot,\la)$ is bounded from below on $[t_\la,\infty)$. Consequently, Theorem \ref{thm:ess.abs} applies and yields 
\[
\sesssing(\cA) = \emptyset, \quad \sess(\cA)= \sessreg(\cA) = [0,1].
\]
However, since $t_\la$ is a regular endpoint and $V(\cdot,\la)$ is bounded from below on $[t_\la,\infty)$, $S(\la)$ is in limit-point case at $\infty$, see \cite[Proposition~4.8.9]{BEH}, but $\sesssing(\cA) = \emptyset$.
\end{example}      
\subsection{Description of $\sesssing(\cA) \setminus \Lambda_\beta(D) \neq \emptyset$} 
The following result characterizes the singular part of the essential spectrum in terms of the limits of some functions formed out of the coefficients of the original operator matrix $\cA_0$ at the singular endpoint $\beta$. For the proof, we need the following~assumptions. 

\begin{ass}\label{ass:pres}
Suppose that
\vspace{1mm}

\begin{enumerate}[label={{\upshape(D\arabic{*})}}]
\item \label{ass:pres.reg} $p\in C^2(\citv,\RR)$,  $b\in C^2(\citv,\CC^n)$ and $D\in C^2(\citv,\CC^{n\times n})$;
\vspace{2mm}

\item \label{ass:pres.pi} 
for every $\la\in \RR \setminus \big( \sessreg(\cA) \cup \Lambda_\beta(D) \big)$, there exists $ \wt t_\la \!\in\! \citvla$ with $t_\la$ as in Assumption~\ref{ass:Schur}, such that 
\begin{equation}
\wt{\pi}(\mydot,\la),  \ \frac{1}{\wt{\pi}(\mydot,\la)} \ \text{  are bounded on } \itvlat,
\end{equation}
where 
\begin{align}\label{wtpi}
\wt{\pi}(t,\la) := 
\begin{cases} 
\dfrac{\pi(t,\la)}{(\beta-t)^2} & \mbox{if } \beta<\infty, \\
\hspace{1.5mm} \pi(t,\la) & \mbox{if } \beta=\infty;
\end{cases}
\end{align}
\vspace{1mm}
\item \label{ass:pres.lim} 
\
for every $\la\in \RR \setminus \big( \sessreg(\cA) \cup \Lambda_\beta(D) \big)$, the limits 
\[
\wt r_{\beta}(\lambda) := \lim_{t\nearrow\beta} \wt{r}(t,\la),
\qquad \wt\varkappa_{\beta}(\lambda) :=\lim_{t\nearrow\beta} \wt{\varkappa}(t,\la)
\]
exist and are finite where 
\begin{alignat*}{2}
\wt{r}(t,\la) &:= 
\left\{
\begin{array}{rl} 
\!\!(\beta-t)\dfrac{r(t,\la)}{\pi(t,\la)} & \mbox{if } \beta<\infty, \\
 \dfrac{r(t,\la)}{\pi(t,\la)} & \mbox{if } \beta=\infty, 
\end{array}
\right.
& \quad
\wt{\varkappa}(t,\la) &:= 
\left\{
\begin{array}{rl} 
\!\!(\beta-t)^2\dfrac{\varkappa(t,\la)}{\pi(t,\la)} & \mbox{if } \beta< \infty, \\
\dfrac{\varkappa(t,\la)}{\pi(t,\la)} & \mbox{if } \beta= \infty;
\end{array}
\right. \hspace{-7mm}
\intertext{moreover, assume that the functions }
\Phi_1(t,\la) &:=
\left\{
\begin{array}{rl} 
\!\!(\beta-t)\dfrac{\frac{\partial}{\partial t}\pi(t,\la)}{\pi(t,\la)} & \mbox{if } \beta< \infty, \\
 \dfrac{\frac{\partial}{\partial t}\pi(t,\la)}{\pi(t,\la)} & \mbox{if } \beta= \infty, 
\end{array}
\right.
& \quad
\Phi_2(t,\la) &:=
\left\{
\begin{array}{rl} 
\!\!(\beta-t)^2\dfrac{\frac{\partial}{\partial t}r(t,\la)}{\pi(t,\la)} & \mbox{if } \beta< \infty, \\
 \dfrac{\frac{\partial}{\partial t}r(t,\la)}{\pi(t,\la)} & \mbox{if } \beta= \infty, 
\end{array}
\right. \hspace{-7mm}
\end{alignat*}
have finite limits as $t\nearrow \beta$.
\end{enumerate}
\end{ass}
Assumption \ref{ass:pres} has the following important consequence. The proof of the following lemma, which is based on Gronwall's inequality, is analogous to the proof of \cite[Lemma~4.2]{ILLT13} and thus omitted.
\begin{lemma}
\label{lem:pres}
Let Assumption \ref{ass:pres} be satisfied. Then, for every 
$\la\in \RR \setminus \big( \sessreg(\cA) \cup \Lambda_\beta(D) \big)$,
we have
\begin{align*}
\lim_{t\nearrow\beta}\Phi_1(t,\la) = 
\begin{cases}
-2 & \mbox{if } \ \beta< \infty, \\
\ 0  & \mbox{if } \ \beta= \infty,
\end{cases}
\qquad
\lim_{t\nearrow\beta}\Phi_2(t,\la) = 
\begin{cases}
-\wt r_{\beta}(\la) & \mbox{if } \ \beta< \infty \\
\ \ \ 0               & \mbox{if } \ \beta= \infty.
\end{cases}
\end{align*}
\end{lemma}

The following theorem does not only generalize \cite[Theorem~4.3]{ILLT13} to the matrix case $n \geq 1$, but it requires considerably weaker assumptions.
%
%
\begin{theorem} 
\label{thm:ess.pres}
Let Assumptions \ref{ass:regularity}, \ref{ass:Schur}, and \ref{ass:pres} be satisfied. Then, for every closed symmetric extension $\cA$ of $\cA_0$ defined in \eqref{A02}, 
\[
  \sess(\cA) \setminus \Lambda_\beta(D)
  = \big( \sessreg (\cA)\cup \sesssing (\cA)   \big) \, \setminus \Lambda_\beta(D)
\]
and
\begin{align*}
  \sesssing(\cA) \setminus \Lambda_\beta(D)
  &= \Bigl\{\la\in\RR \setminus \big( \sessreg (\cA) \cup \Lambda_\beta(D) \big) : \, \Dis_{\beta}(\la) \geq 0 \Bigr\} 
\end{align*} 
where
\begin{align}\label{Discriminant}
\Dis_{\beta}(\la) :=
\begin{cases} 
\wt{r}_{\beta}(\la)^2 - \wt{\varkappa}_{\beta}(\la) - \dfrac{1}{4} & \mbox{if } \ \beta< \infty, \\
\wt{r}_{\beta}(\la)^2 - \wt{\varkappa}_{\beta}(\la) & \mbox{if } \ \beta= \infty.
\end{cases}
\end{align}
\end{theorem}

\begin{proof}
The differential expression $\tau_S(\la)$ in \eqref{Ssymm} can be rewritten as 
\begin{equation}\label{S.non.symm}
\tau_S(\la) 
=
-\pi(\cdot,\la)\dfrac{\d^2}{\d t^2} +\rho(\cdot,\la)\, \ii\dt + \kappa(\cdot,\la)
\end{equation}
where 
\begin{equation}\label{rho.kappa.def}
\begin{aligned}
\kappa(\cdot,\la) 
:= \varkappa(\cdot,\la) +\ii\frac{\partial}{\partial t}r(\cdot,\la),
\quad
\rho(\cdot,\la)  
:= 2r(\cdot,\la) + \ii\frac{\partial}{\partial t}\tpi(\cdot,\la).
\end{aligned}
\end{equation}
Let $\wt t_\la \in[t_\la,\beta)$ be as in Assumption \ref{ass:pres}. 
By Lemma \ref{lem:AS1}, 
\begin{equation}\label{sess-equiv}
\la \in \sess(\cA) \setminus (\sessreg(\cA) \cup \Lambda_\beta(D) ) \  \iff \  0 \in \sess(S_{(\wt t_\la,\beta)}(\la)).
\end{equation}

To analyse when $0 \in \sess(S_{(\wt t_\la,\beta)}(\la))$, we use \cite[Corollary~IX.9.4]{EE87}. We distinguish the cases $\beta< \infty$  and $\beta= \infty$. If $\beta<\infty$, we consider the unitary transformation  
\begin{equation*}
U : L^2(\itv) \to L^2((\alpha,\infty)), \quad (Uf)(x) := \psi(x)f(\varepsilon(x)),
\end{equation*}
with
\[
\varepsilon(x) := (\alpha-\beta)\e^{-(x-\alpha)}+\beta, \qquad \psi(x) := \sqrt{\beta-\alpha} \, \e^{-\frac{1}{2}(x-\alpha)}, \quad x\in [\alpha,\infty).
\]
Then $\varepsilon$ maps $[\alpha,\infty)$ bijectively onto $[\alpha,\beta)$ and
\[
U\dx U^{-1} = \frac{1}{\beta-\varepsilon(\cdot)}\dx + \frac{1}{2}\frac{1}{\beta-\varepsilon(\cdot)}.
\]
Now it is not difficult to check that Assumption \ref{ass:pres} and Lemma \ref{lem:pres} ensure that, for fixed $\la$, 
\[
\ds\frac{(\beta-\varepsilon(\cdot))^2}{\tpi(\varepsilon(\cdot), \la)}US_{(\wt t_\la,\beta)}(\la)U^{-1}
\]
satisfies the assumptions of \cite[Corollary~IX.9.4]{EE87}\footnote{Note that it is enough to require $a_1'\in L^{\infty}_{\rm loc}(I)$ in \cite[p.\ 445, (iii)]{EE87} for \cite[Corollary~IX.9.4]{EE87}, cf.\ \cite[p.\ 437 top]{ILLT13}.} with $m=2$, $a_2=-1$, $a_1 = 2\wt r_{\beta}(\la)$, and $a_0 = \wt\varkappa_{\beta}(\la)+\frac{1}{4}$. Hence \cite[(9.19)]{EE87} with $k=3$ applies and yields 
\[
  0 \in \sess(S_{(\wt t_\la,\beta)}(\la)) = \sess\Bigg(\frac{(\beta-\varepsilon(\cdot))^2}{\tpi(\varepsilon(\cdot), \la)}US_{(\wt t_\la,\beta)}(\la)U^{-1}\Bigg) 
  \ \iff \
  \exists\, \xi \in \RR : \xi^2 + 2\wt r_{\beta}(\la) \xi + \wt\varkappa_{\beta}(\la)+\frac{1}{4} =0, 
\]
which is, in turn, equivalent to $4\Dis_{\beta}(\la) = 4\wt r_{\beta}(\la)^2-4\wt\varkappa_{\beta}(\la)-1\ge 0$. The claim now follows from Lemma \ref{lem:AS1}.

If $\beta=\infty$, no unitary transform is needed. Our assumptions ensure that $S_{(\wt t_\la,\infty)}(\la)$ itself satisfies the assumptions of \cite[Corollary~IX.9.4]{EE87}. The proof can be finished in the same way as above. 
\end{proof}

\section{Sufficient conditions for Assumptions \ref{ass:regularity} to \ref{ass:pres}} 
\label{subsec:suff}

In this section, we derive sufficient conditions for the assumptions of Lemma \ref{lem:AS1}, Theorem \ref{thm:ess.abs}, and Theorem~\ref{thm:ess.pres} that are easier to verify in applications.
We restrict ourselves to the case when the singular endpoint is finite, $\beta<\infty$, since the case $\beta=\infty$ can be easily transformed to the finite interval case.

To this end, in Assumption \ref{ass:suff} below, we will assume that, 
near the singular endpoint $\beta$, the function $\tpi(\cdot,\la)$ has the asymptotic expansion 
\begin{align}\label{pi_asymp}
\tpi(t,\la) = \tpi_0(\la) + \tpi_1(\la)(t-\beta) + \mathcal{R}(t,\la) \quad \text{with} \quad \cR(t,\la)=\o(\beta-t), \ t\nearrow\beta.
\end{align}
Since $\tpi(\cdot,\la) \in C^1([\alpha,\beta),\RR)$ by Assumption \ref{ass:regularity}, we have 
\[
 \tpi_0(\la)=\lim_{t\nearrow\beta}\tpi(t,\la), \quad \tpi_1(\la)=\lim_{t\nearrow\beta}\frac{\partial}{\partial t}\tpi(t,\la). 
\]
To describe the singular part of the essential spectrum, we distinguish the following possible cases.
\begin{align}\label{Cases}
\nonumber
&\textbf{Case (I)}:   &&\tpi_0(\la) \neq 0;                   &  \hspace{550mm}
\\
&\textbf{Case (II)}:  &&\tpi_0(\la) = 0,  \ \tpi_1(\la) \neq 0; & 
\\ \nonumber
&\textbf{Case (III)}: &&\tpi_0(\la) =0, \ \tpi_1(\la) = 0.      &  
\end{align}
\setcounter{ass}{18}
\begin{ass} \label{ass:suff}
Suppose that, for every $\la \in \RR \setminus \big( \sessreg(\cA) \cup \Lambda_\beta(D) \big)$ and  $t\nearrow\beta$,
\begin{align}
\label{ass:S.Rb}
\mathcal{R}(t,\la) & = 
\begin{cases} 
\ds \o(\beta-t) 
& \text{in Case (I)},
\\
\ds \o \left( \frac{\beta-t}{|\log(\beta-t)|} \right) 
& \text{in Case (II)},
\\
\ds \O ( (\beta-t)^2 ) 
& \text{in Case (III)},
\end{cases}%
\end{align}
\begin{align}
\hspace{-3mm}
\|(D(t)\!-\!\la)^{-1}b(t)\|^2_{\CC^n} \!=\! 
\begin{cases} 
\ds \O(1)
&\!\! \text{in Case (I)},
 \\
\ds \o(\beta-t)
 &\!\! \text{in Case (II)},
\\
\ds \O((\beta\!-\!t)^2)
&\!\! \text{in Case (III)},
\end{cases}
\ \
\label{ass:S.r}
r(t,\la)\!=\! 
\begin{cases} 
\ds \o \left(\frac 1{\beta-t}\right) 
&\!\! \text{in Case (I)},
\\[2.5mm]
\ds \o \left( \frac 1{|\log(\beta\!-\!t)|} \right)
&\!\! \text{in Case (II)},
\\[3.1mm]
\O(\beta-t)
&\!\! \text{in Case (III)},
\end{cases}
\end{align}
and, with $s_\pi =\sgn (\tpi(\cdot,\la)\!\!\restriction\! _{\citvla})$ defined as in \eqref{spi.def} and $f_-:=(f-|f|)/2$ denoting the negative part of a real-valued function ~$f$, 
\begin{equation}
\begin{aligned}
\label{ass:S.kappac}
\exists \ \varepsilon >0: \ \Big(s_\pi \varkappa(t,\la) - \varepsilon \|(D(t)-\la)^{-1}c(t)\|^2_{\CC^n}
\Big)_{-}  &= 
\begin{cases} 
\ds \o \left(\frac 1{(\beta-t)^{-2}}\right)
& \text{in Case (I)},
\\[2.5mm]
\ds\o \left(\frac 1{ (\beta-t) \log(\beta-t)^{2}}\right) 
& \text{in Case (II)},
\\[3.1mm]
\O(1)
& \text{in Case (III)}.
\end{cases}
\end{aligned}
\end{equation}
\end{ass}

\begin{theorem}\label{prop:suff}
Suppose that $\beta< \infty$ and Assumptions \ref{ass:regularity}, \ref{ass:suff} are satisfied. Then the following~hold.

\begin{enumerate}[label=\rm{(\roman{*})}]
\item \label{prop:suff.I} In Cases {\rm (I)} and {\rm (II)}, Assumptions \ref{ass:Schur} and \ref{ass:abs} are satisfied. Hence, for every closed symmetric extension $\cA$ of $\cA_0$ defined in \eqref{A02},  
\begin{equation}
\sesssing(\cA) \setminus \Lambda_\beta(D) = \emptyset, \qquad \text{i.e. \ \ }
\sess(\cA) \setminus \Lambda_\beta(D)
= \sessreg(\cA) \, \setminus \Lambda_\beta(D).
\end{equation}
\item \label{prop:suff.III} In Case {\rm (III)}, if additionally Assumption \ref{ass:pres.reg} holds and, for every $\la \in \RR \setminus \big( \sessreg(\cA) \cup \Lambda_\beta(D) \big)$, the following limits exist and satisfy
\begin{align}\label{suff.lim.pi.r.kappa}
\hspace{2mm}
\pi_2(\la) \!:=\!\frac 12 \lim_{t\nearrow\beta}\frac{\partial^2}{\partial t^2}\tpi(t,\la) \in \RR \!\setminus\! \{0\}, 
\ \ 
\ds r_1(\la) \!:=\! \lim_{t\nearrow\beta}\frac{\partial}{\partial t}r(t,\la) \in \RR, 
\ \ 
\varkappa_0(\la) \!:=\! \lim_{t\nearrow\beta}\varkappa(t,\la) \in \RR,
\hspace{-8mm}
\end{align}
then Assumptions \ref{ass:Schur}, \ref{ass:pres.pi}--\ref{ass:pres.lim} are satisfied. Hence, for every closed symmetric extension $\cA$ of~$\cA_0$ defined in \eqref{A02}, 
\[
\sess(\cA) \setminus \Lambda_\beta(D)
= \big( \sessreg (\cA)\cup \sesssing (\cA)   \big) \, \setminus \Lambda_\beta(D)
\]
where
\begin{equation}\label{sing_part_suff}
\sesssing ( \cA ) \setminus\Lambda_{\beta}(D) = \left\{ 
\la \in \RR \setminus \big( \sessreg(\cA) \cup \Lambda_\beta(D) \big) :  
r_1(\la)^2 - \varkappa_0(\la)\pi_2(\la)  \geq \frac 14 \pi_2(\la)^2
\right\}.
\end{equation} 
\end{enumerate}
\end{theorem}

\begin{proof} \ref{prop:suff.I} \
The asymptotic behaviour of $\|(D-\la)^{-1}b\|_{\CC^n}^2$, see \eqref{ass:S.r}, implies that \eqref{pi.B1a} holds with some constant $k_1>0$, i.e.\   Assumption \ref{ass:Schur.1} is satisfied. 

Below, we verify that Assumption \ref{ass:Schur.2} is satisfied. 
To this end, let $\eta \in C^2(\citvla)$ be positive.
The triangle inequality, condition \eqref{pi.B1a}, and the Cauchy-Schwartz inequality yield
\begin{equation}\label{Veta.suff0}
\begin{aligned}
\bigg\| 
(D\!-\!\la)^{-1}\Bigl[\Bigl(\ii\frac{r(\cdot,\la)}{\tpi(\cdot,\la)}+\frac{\eta'}{2\eta}\Bigr)b+c\Bigr] \bigg\|^2_{\CC^n} 
\!\!&\! \leq 
\Bigg[
\Bigl(\Bigl|\frac{r(\cdot,\la)}{\tpi(\cdot,\la)}\Bigr| \!+\! \Bigl| \frac{\eta'}{2\eta} \Bigr| \Bigr)\|(D\!-\!\la)^{-1}b\|_{\CC^n} \!+\! \|(D\!-\!\la)^{-1}c\|_{\CC^n}
\Bigg]^2
\\
& \leq 
\Bigg[
\Bigl(\Bigl|\frac{r(\cdot,\la)}{\tpi(\cdot,\la)}\Bigr| + \Bigl| \frac{\eta'}{2\eta} \Bigr| \Bigr) \sqrt{\frac{|\tpi(\cdot,\la)|}{k_1}} + \|(D-\la)^{-1}c\|_{\CC^n}
\Bigg]^2
\\
&\leq 
\frac{3}{k_1}\frac{r(\cdot,\la)^2}{|\tpi(\cdot,\la)|} 
+ \frac{3}{k_1} \Bigl( \frac{\eta'}{2\eta} \Bigr)^2 |\tpi(\cdot,\la)|
+ 3 \|(D-\la)^{-1}c\|^2_{\CC^n}.
\end{aligned}
\end{equation}
Using the definition of $V_\eta$ in \eqref{V_eta} and the estimate \eqref{Veta.suff0} above, we obtain
\begin{equation}\label{Veta.suff}
\begin{aligned}
& V_{\eta}(\cdot,\la) 
- k_2 \bigg\|
(D-\la)^{-1}\Bigl[\Bigl(\ii\frac{r(\cdot,\la)}{\tpi(\cdot,\la)}+\frac{\eta'} {2\eta}\Bigr) b + c\Bigr]
\bigg\|^2_{\CC^n} 
\\
& \geq 
s_\pi \!
\left( \!
\varkappa(\cdot,\la)
- k_3 \frac{r(\cdot,\la)^2}{\tpi(\cdot,\la)} 
+ 
\left( \! k_4 \!
\left( \frac{\eta'}{2\eta}\right)^2 \!\!
 - \frac{\eta''}{2\eta}
\right)
\tpi(\cdot,\la) 
- \frac{\eta'}{2 \eta}  \frac{\partial}{\partial t} \tpi(\cdot,\la) 
- \frac{3 k_2}{s_\pi} \|(D-\la)^{-1}c\|^2_{\CC^n} \!
\right)
\end{aligned}
\end{equation}
where $k_3:= 1+\dfrac{3k_2}{k_1}$, $k_4: = 1 - \dfrac{3k_2}{k_1}$ and we choose $k_2>0$ such that $k_2<\min\Bigl\{\dfrac{\varepsilon}{3}, \dfrac{k_1}{3}\Bigr\}$.

In Case (I), we choose $\eta(t) := \beta -t$, $t\in[\alpha,\beta)$. Then the right-hand side of \eqref{Veta.suff} becomes
\begin{equation}
		\begin{aligned}
		& s_\pi \!
		\left( \! \varkappa(t,\la)
		\!+\! \frac{1}{2(\beta-t)} \frac{\partial}{\partial t} \tpi(t,\la)   
		\!-\! k_3 \frac{r(t,\la)^2}{\pi(t,\la)} 
		\!+\! \frac{k_4}{4 (\beta-t)^2} \pi(t,\la) 
		\!-\! \frac{3 k_2}{s_\pi} \|(D(t)-\la)^{-1}c(t)\|^2_{\CC^n}
		\right)
		\\
		&\geq
		\frac{|\pi_0(\la)|}{(\beta-t)^2}
		\left( 
		\frac{k_4}{4} \frac{\pi(t,\la)}{\pi_0(\la)}
		+ \frac{\beta-t}{2 \pi_0(\la)} \frac{\partial}{\partial t} \tpi(t,\la)   
		- k_3 \frac{r(t,\la)^2 (\beta-t)^2}{|\pi(t,\la)| |\pi_0(\la)|} 
		\right.
		\\
		& \qquad \qquad \qquad \qquad 
		\left.
		+ \frac{(\beta-t)^2}{|\pi_0(\la)|}
		\Big(
		s_\pi \varkappa(t,\la) - \varepsilon \|(D(t)-\la)^{-1}c(t)\|^2_{\CC^n}
		\Big)
		\right)
		.
		\end{aligned}
		\end{equation}
Now Assumption \ref{ass:suff}, i.e.\ the asymptotic conditions \eqref{ass:S.Rb}--\eqref{ass:S.kappac}, yield that
\begin{equation}\label{Veta.suff.I}
\begin{aligned}
V_{\eta}(\cdot,\la) 
- k_2 \bigg\|
(D-\la)^{-1}\Bigl[\Bigl(\ii\frac{r(\cdot,\la)}{\tpi(\cdot,\la)}+\frac{\eta'} {2\eta}\Bigr)b+c\Bigr]
\bigg\|^2_{\CC^n} 
\geq 
\frac{k_4}{4} \frac{|\pi_0(\la)|}{(\beta-t)^2}
\left(1  + \o(1) 
\right), \quad  t\nearrow\beta, 
\end{aligned}
\end{equation}
thus Assumption \ref{ass:Schur.2} is satisfied.

In Case (II), we choose $\eta(t):= - \log (C (\beta-t))$, $t\in[\alpha,\beta)$, with $C^{-1}:=2 (\beta-\alpha)$ so that $\eta$ is positive on $\citv$. Then the right-hand side of \eqref{Veta.suff} becomes
\begin{equation}
		\begin{aligned}
		\hspace{-6mm} & \hspace{-6mm} s_\pi 
		\Bigg( 
		\frac{k_4}{4} \frac{\pi(t,\la)}{(\beta-t)^2 \log^2(C(\beta-t))} 
		\left(
		1 + \frac{2}{k_4}\frac{(\pi(t,\la) 
			+ (\beta-t)\frac{\partial}{\partial t} \tpi(t,\la))\log(C(\beta-t))}{\pi(t,\la)}
		\right) 
		\\
		\hspace{-6mm} & 
		- k_3 \frac{r(t,\la)^2}{\pi(t,\la)} 
		+ \varkappa(t,\la) - \frac{3 k_2}{s_\pi} \|(D(t)-\la)^{-1}c(t)\|^2_{\CC^n}
		\Bigg)
		\\
		\hspace{-6mm}  & \hspace{-6mm} \geq 
		\frac{|\pi_1(\la)|}{(\beta-t) \log^2(C(\beta-t))}
		\Bigg( 
		\frac{k_4}{4} \frac{\pi(t,\la)}{\pi_1(\la)(t-\beta)}
		\left(
		1 + \frac{2}{k_4}\frac{(\pi(t,\la) 
			+ (\beta-t)\frac{\partial}{\partial t} \tpi(t,\la)) \log(C(\beta-t))}{\pi(t,\la)}
		\right)
		\\ 
		\hspace{-6mm} & \qquad \qquad \qquad \qquad \qquad 
		- k_3 \frac{r(t,\la)^2 (\beta-t) \log^2(C(\beta-t)) }{|\pi(t,\la)| |\pi_1(\la)|} 
		\\
		&  \qquad \qquad \qquad \qquad \qquad  
		+
		\frac{(\beta-t)\log^2(C(\beta-t))}{|\pi_1(\la)|}
		\Big( s_\pi\varkappa(t,\la) - \varepsilon\|(D(t)-\la)^{-1}c(t)\|^2_{\CC^n}
		\Big)
		\Bigg) 
		.
		\end{aligned}
		\end{equation}
Now Assumption \ref{ass:suff}, i.e.\ the asymptotic conditions  \eqref{ass:S.Rb}--\eqref{ass:S.kappac}, yield that
\begin{equation}\label{Veta.suff.II}
\begin{aligned}
V_{\eta}(\cdot,\la) 
- k_2 \bigg\|
(D-\la)^{-1}\Bigl[\Bigl(\ii\frac{r(\cdot,\la)}{\tpi(\cdot,\la)}+\frac{\eta'} {2\eta}\Bigr)b+c\Bigr]
\bigg\|^2_{\CC^n} 
\geq 
\frac{k_4}{4} \frac{|\pi_1(\la)|}{(\beta-t)\log^2(C(\beta-t))}
\left(1  + \o(1) 
\right), \quad  t\nearrow\beta, 
\end{aligned}
\end{equation}
thus Assumption \ref{ass:Schur.2} is satisfied.

It remains to verify Assumption \ref{ass:abs}, i.e.\ that either \ref{ass:C1} or \ref{ass:molchanov} hold. 
Note that the choices $\eta(t) := \beta -t$, $t\in [\alpha,\beta)$, in Cases (I) and $\eta(t):= - \log (C (\beta-t))$, $t\in [\alpha,\beta)$, in Case (II), respectively, lead to
\begin{align}\label{eta_pi}
\eta(t)\tpi(t,\la)=
\begin{cases}
\tpi_0(\la)(\beta-t)(1 + \o(1)) & \text{in Case (I)},
\\
\tpi_1(\la)(\beta-t) \log (C (\beta-t))(1 + \o(1)) & \text{in Case (II)},
\end{cases}
\quad  t\nearrow\beta.
\end{align}
Hence $\frac 1{\eta\tpi(\cdot,\la)} \notin  L^1( \itvla)$ in both Case (I) and (II). Then we are in the second case of \eqref{h.def}.
It is not difficult to see that $\eta(t)h(t,\la^2 = \o((\beta-t)^{-\delta})$ for arbitrary $\delta>0$ as $ t \nearrow \beta$, thus $ \eta h(\cdot,\la)^2 \in L^1( \itvla)$ and hence \ref{ass:C1} holds.

\smallskip
\ref{prop:suff.III} \ In Case (III), we first note that Assumption \ref{ass:pres.reg} implies $\pi(\cdot,\la) \in C^2(\citv,\RR)$. As $\pi_0(\la) = \pi_1(\la)=0$ and $r(t,\la)=\O(\beta-t)$, $t\nearrow\beta$, the existence of the limits in \eqref{suff.lim.pi.r.kappa} and L'H\^opital's rule yield that the following limits exist and satisfy
\begin{align*}
\lim_{t\nearrow\beta}\frac{\tpi(t,\la)}{(\beta-t)^2} &=
\frac 12 \lim_{t\nearrow\beta}\frac{\partial^2}{\partial t^2}\tpi(t,\la) = \pi_2(\la), 
\qquad
\lim_{t\nearrow\beta}(\beta-t)\frac{\frac{\partial}{\partial t}\tpi(t,\la)}{\tpi(t,\la)} 
= 
\frac{1}{\pi_2(\la)} \lim_{t\nearrow\beta}\frac{\frac{\partial}{\partial t}\tpi(t,\la)}{\beta-t} 
= 
-1,
\\
\lim_{t\nearrow\beta}\frac{r(t,\la)}{t - \beta} &=
\lim_{t\nearrow\beta}\frac{\partial}{\partial t}r(t,\la) = r_1(\la).
\end{align*}
Since $\pi_2(\la) \ne 0$ by assumption \eqref{suff.lim.pi.r.kappa}, Assumptions \ref{ass:pres.pi}--\ref{ass:pres.lim} in the case $\beta< \infty$ considered here are satisfied. Moreover, due to the asymptotic behaviour of $\|(D-\la)^{-1}b\|_{\CC^n}^2$, see \eqref{ass:S.r}, and $\pi_2(\la)\neq 0$, it follows that Assumption \ref{ass:Schur.1} is satisfied as well.

To verify Assumption \ref{ass:Schur.2}, we proceed as in \eqref{Veta.suff0}--\eqref{Veta.suff} choosing $\eta(t):= 1$, $t\in[\alpha,\beta)$. The asymptotic conditions \eqref{ass:S.Rb}--\eqref{ass:S.kappac} and $\pi_2(\la) \neq 0$ yield that
\begin{equation}\label{Veta.suff.III}
\begin{aligned}
V_{\eta}(\cdot,\la) 
- k_2 \bigg\|
(D-\la)^{-1}\Bigl[ \ii\frac{r(\cdot,\la)}{\tpi(\cdot,\la)} b+c\Bigr]
\bigg\|^2_{\CC^n} 
\end{aligned}
\end{equation}
is bounded from below on some left-neighbourhood of $\beta$. Using the relations \eqref{suff.lim.pi.r.kappa}, we obtain
\begin{align*}
\wt{r}_{\beta}(\la) &= \lim_{t\nearrow\beta}(\beta-t)\frac{r(t,\la)}{\tpi(t,\la)} 
= - \lim_{t\nearrow\beta} \frac{r(t,\la)}{t-\beta} \lim_{t\nearrow\beta}\frac{(\beta-t)^2}{\tpi(t,\la)} = -\frac{r_1(\la)}{\tpi_2(\la)},
\\
\wt{\varkappa}(\la) &= \lim_{t\nearrow\beta}(\beta-t)^2\frac{\varkappa(t,\la)}{\tpi(t,\la)} 
= 
\lim_{t\nearrow\beta}\varkappa(t,\la)\lim_{t\nearrow\beta}\frac{(\beta-t)^2}{\tpi(t,\la)} = \frac{\varkappa_0(\la)}{\tpi_2(\la)} .
\end{align*}
Therefore, by definition \eqref{Discriminant},
\[
\Dis_{\beta}(\la) = \left(
\frac{r_1(\la)}{\pi_2(\la)}
\right)^2
-
\frac{\varkappa_0(\la)}{\pi_2(\la)}
-\frac{1}{4}
\]
Multiplying the inequality $\Dis_\beta(\la) \ge 0$ by $\tpi_2(\la)^2>0$,
Theorem \ref{thm:ess.pres} yields the desired description of $\sesssing ( \cA ) \setminus\Lambda_{\beta}(D)$ in 
\eqref{sing_part_suff}. 
\end{proof}

\begin{remark}\label{rem:SII}
In Case (II), the claim of Theorem \ref{prop:suff} \ref{prop:suff.I} continues to hold if, instead of the asymptotic conditions in Assumption \ref{ass:suff}, the asymptotic conditions in the following Assumption \ref{ass:suff'} hold.

\setcounter{assu}{18}
\begin{assu} \label{ass:suff'}
Suppose that, in Case (II), for every $\la \in \RR \setminus \big( \sessreg(\cA) \cup \Lambda_\beta(D) \big)$ and $t\nearrow\beta$,
		\begin{align}
		\label{rem:SII.Rb}
		& \mathcal{R}(t,\la)  = \o(\beta-t), \\
		\label{rem:SII.D}
		& \|(D(t)-\la)^{-1}b(t)\|^2_{\CC^n}  = \O(\beta-t), 
		\\ 
		\label{rem:SII.r}
		& r(t,\la) = \O((\beta-t)^{1/2}), 
		\\
		\label{rem:SII.kappac}
		& \exists \ \varepsilon>0: \ \left( s_\pi \varkappa(t,\la) - \varepsilon \|(D(t)-\la)^{-1}c(t)\|^2_{\CC^n} 
		\right)_-  = \O(1); 
		\end{align}
		here, again, $s_\pi=\sgn(\tpi(\cdot,\la)\!\!\restriction\!_{\citvla})$ is as in \eqref{spi.def} and $f_-:=(f-|f|)/2$ is the negative part of a real-valued function~$f$.
\end{assu}
\end{remark}
\begin{proof}
The assumption on the behaviour of $\|(D-\la)^{-1}b\|^2_{\CC^n}$ is the same as in Assumption \ref{ass:suff}, thus Assumption \ref{ass:Schur.1} is satisfied with some constant $k_1>0$, see the proof of Theorem \ref{prop:suff}. 

Choosing $\eta(t):=1$ and using \eqref{rem:SII.Rb}--\eqref{rem:SII.kappac}, we obtain from \eqref{Veta.suff} that
\[
\begin{aligned}
V_{\eta}(\cdot,\la) 
- k_2 \bigg\|
(D-\la)^{-1}\Bigl[ \ii\frac{r(\cdot,\la)}{\tpi(\cdot,\la)} b+c\Bigr]
\bigg\|^2_{\CC^n} 
\end{aligned}
\]
is bounded from below on some left-neighbourhood of $\beta$ and thus Assumption \ref{ass:Schur.2} is satisfied. 

Moreover, $\eta(t)\tpi(t,\la)= - \tpi_1(\la)(\beta-t)(1 + \o(1))$, $t\nearrow\beta$, so Assumption \ref{ass:C1} is satisfied by the same arguments as in the proof of Theorem \ref{prop:suff}.
\end{proof}
%

\section{The structure of the singular part of the essential spectrum}
\label{sec:str.of.sing.part}

In this section, we analyse the topological structure of the essential spectrum. We start with a simple observation on the regular part of the essential spectrum.

\begin{proposition}
\label{prop:reg.part.structure}
The regular part $\sessreg(\cA)=\overline{\Lambda_{\citv}(\Delta)}$ of the essential spectrum is the union of at most $n$ closed intervals in $\RR$ which are the closures of the ranges of the eigenvalues of the matrix function $\Delta=D-\frac{1}{p}bb^*$ in \eqref{defDelta}. Moreover, 
\begin{align}
\label{ineq:reg.1}
\inf\sessreg(\cA) \leq \inf \Lambda_{\beta}(D).
\end{align}
\end{proposition}
\begin{proof}
The first claim is obvious from the definition of the regular part of the essential spectrum since $\Delta$ is a Hermitian matrix-valued function which is continuous on $[\alpha,\beta)$.
To prove the second claim, note that, for every $t\in[\alpha,\beta)$, we have 
$\Delta(t) = (D-\frac{1}{p}bb^*) (t) \leq D(t)$  in the sense of partial operator ordering since $p>0$ and hence
\begin{align*}
\min \spt(\Delta(t)) = \underset{\|x\|_{\CC^n}=1}{\min} \bigl(\Delta(t)x,x\bigr)_{\CC^n} &= \underset{\|x\|_{\CC^n}=1}{\min} \Big(\bigl(D(t)x,x\bigr)_{\CC^n} - \frac{1}{p(t)}\|b(t)^*x\|^2_{\CC^n}\Big) \\
                        &\leq \underset{\|x\|_{\CC^n}=1}{\min}\bigl( D(t)x,x\bigr)_{\CC^n} = \min\spt(D(t)).
\end{align*}
Hence, by \eqref{Lam.J.def}, we obtain
\begin{align*}
\inf\sessreg(\cA) = \inf\ov{\Lambda_{[\alpha, \beta)}(\Delta)} = \underset{t \in [\alpha,\beta)}{\inf}\bigl(\min\spt(\Delta(t))\bigr) \leq \underset{t \in [\alpha,\beta)}{\inf}\bigl(\min\spt(D(t))\bigr) = \inf\ov{\Lambda_{[\alpha, \beta)}(D)} \leq \inf\Lambda_{\beta}(D)
\end{align*}
where we have used $\Lambda_{\beta}(D) \subset \ov{\bigcup_{t\in[\alpha,\beta)} \spt(D(t))}$ in the last step.
\end{proof}

In the sequel, we continue with the analysis of the singular part of the essential spectrum. Here we use that the leading coefficient  
of the first Schur complement has the property that $\la \mapsto -\tpi(t,\la)$ is a Nevanlinna function of $\la\in\CC$ for all sufficiently large $t$. 

The class of Nevanlinna functions consist of those complex functions that are analytic on the open upper half-plane and have nonnegative imaginary part therein (cf.\ \cite{KK1}). 
It is well-known that a function $f$ is a Nevanlinna function if and only if it admits a canonical integral representation of the form (cf.\ \cite{KK1}) 
\begin{align}
\label{nev}
f(\zeta)=\omega_1+\omega_2 \zeta+\int_{\RR}\Bigl(\frac{1}{\nu-\zeta}-\frac{\nu}{\nu^2+1}\Bigr)\text{d}\sigma(\nu), \quad \zeta \in \CC \setminus {\rm supp\,} \sigma,
\end{align}
with $\omega_1, \omega_2 \in \RR$, $\omega_1\geq 0$, and a positive Borel measure $\sigma$ on $\RR$ such that 
\begin{align*}
\int_{\RR}\frac{\text{d}\sigma(\nu)}{1+\nu^2}< \infty;
\end{align*}
moreover, this representation is unique.
If $f$ is a rational Nevanlinna function, then the corresponding measure $\sigma$ is concentrated at the real poles $\{\nu_j\}_{j=1}^m$ of $f$ and thus the integral representation \eqref{nev} takes the form
\begin{align}\label{rep.of.Nevanlinna.fn}
f(\zeta)=\omega_1\zeta+\omega_2+\sum_{j=1}^m\frac{\sigma_j}{\nu_j-\zeta}, \quad \zeta\in\CC\setminus \{\nu_j\}_{j=1}^m,
\end{align}
with $\sigma_j>0$, $j=1,2,\ldots,m$.

\smallskip
%
%
The following property of the leading coefficient $\tpi(\cdot,\la)$ of the Schur complement plays a crucial role in the description of the structure of the singular part of the essential spectrum.
\begin{lemma}
\label{pika-nev} 
For every $t\!\in\! [\alpha,\beta)$, the function $\zeta\!\mapsto\! -\tpi(t,\zeta)$ is a Nevanlinna function of the complex variable~$\zeta$.
\end{lemma}
\begin{proof}
Let $t\geq 0$ be fixed and denote by $\CC^+\!:=\{z\in\CC : \Im(z)>0\}$ the open upper half-plane in $\CC$. Since $D(t)\in\CC^{n\times n}$ is Hermitian, it has finitely many eigenvalues that are all real. Thus it is clear from the representation \eqref{pidelta} that the function $\zeta\mapsto -\pi(t,\zeta)$ is rational with poles exactly at the eigenvalues of $D(t)$, and therefore holomorphic on $\CC^+$. 

It remains to be shown that $\zeta\mapsto -\pi(t,\zeta)$ maps $\CC^+$ into itself. Denoting $\wh{b}(\zeta):=\bigl(D(t)-\zeta\bigr)^{-1}b(t)$, $\zeta\in\CC^+$, we obtain
\begin{align*}
b(t)^*\bigl(D(t)-\zeta\bigr)^{-1}b(t) = \bigl(\bigl(D(t)-\zeta\bigr)\:\wh{b}(\zeta)\bigr)^*\:\wh{b}(\zeta) & = \wh{b}(\zeta)^*D(t)\:\wh{b}(\zeta)-\overline{\zeta}\:\|\wh{b}(\zeta)\|^2_2.
\end{align*}
Since $p$ is real-valued and $D(t)$ is Hermitian, 
\eqref{pi.r.varkap.def} implies that
\begin{align*}
&
\Im(-\pi(t,\zeta)) = 
- \Im\bigl(p(t)-b(t)^*\bigl(D(t)-\zeta\bigr)^{-1}b(t)\bigr)= \Im(\zeta) \,\|\wh{b}(\zeta)\|^2_2 \ge 0, \quad \zeta \in \CC^+.
\qedhere
\end{align*}
\end{proof}

\vspace{2mm}

Since, for every $t\in [t_\la,\beta)$, the eigenvalues of $D(t)$ denoted by $\la_1(t), \la_2(t), \ldots, \la_n(t) \in \RR$ coincide with the poles of the rational function $\la\mapsto-\pi(t,\la)$, Lemma~\ref{pika-nev} and \eqref{rep.of.Nevanlinna.fn} yield that
\begin{align}
\label{rep:-pi}
-\pi(t,\la) &= -p(t)+\sum_{j=1}^n\frac{\sigma_j(t)}{\la_j(t)-\la}, \quad \la\in \CC\setminus \{\la_j(t)\}_{j=1}^m,
\end{align}
where $\sigma_j$,
$j=1,2,\ldots,n$, are positive functions of $t\in [t_\la,\beta)$.
%
%
\begin{lemma}
\label{nice lemma1}
Let $b_j \in C^1([\alpha,\beta))$, $j=1,2,\ldots,n$, be the coefficients of $\cA_0$ in \eqref{A02}. Then the functions $\sigma_j$, $j=1,2,\ldots, n$, in \eqref{rep:-pi} 
\vspace*{-1mm} satisfy
\begin{align*}
\sum_{j=1}^n\sigma_j(t)=\sum_{j=1}^n|b_j(t)|^2, 
\quad t\in [t_\la,\beta).
\end{align*}
\end{lemma}

\begin{proof}
In the sequel, we work with functions and omit the dependence on $t\!\in\! [t_\la,\beta)$.
By \eqref{pi.r.varkap.def},~\eqref{rep:-pi}, 
\begin{align}
\label{identity}
\sum_{j=1}^n\frac{\sigma_j}{\la_j-\la}=b^*(D-\la)^{-1}b.
\end{align}
On the other \vspace*{-1mm} hand, 
\begin{align}
\label{identity-a}
\sum_{j=1}^n\frac{\sigma_j}{\la_j-\la}=\frac{P(\la)}{\det(D-\la)},
\end{align}
where $P$ is a polynomial of degree $n-1$ in $\la$ with leading coefficient 
\begin{align}
\label{lcP}
(-1)^{n-1}(\sigma_1+\ldots+\sigma_n).
\end{align}
Cramer's rule implies \vspace*{-1mm} that
\begin{align}
\label{cofac}
(D-\la)^{-1}=\frac{1}{\det(D-\la)}(M(\cdot,\la))^{\rm t},
\end{align}
where $M(t,\la)$ is the matrix of cofactors for $D(t)-\la$, $t\in [t_\la,\beta)$. Here, the diagonal entries of $M(\cdot,\la)$ are polynomials of degree $n\!-\!1$ in $\la$, while the off-diagonal entries are polynomials of degree at most $n\!-\!2$~in~$\la$. Hence multiplying \eqref{cofac} by $b$ from the right and by $b^*$ from the left, we obtain a rational function whose denominator is 
$\det(D-\la)$ and whose nominator is a polynomial $Q$ of degree $n-1$ in $\la$ with leading coefficient
\begin{align}
\label{lcQ}
(-1)^{n-1}(|b_1|^2+\ldots+|b_n|^2).
\end{align}
Now \eqref{identity}, \eqref{identity-a} imply that $P$ and $Q$ coincide, and hence so do the leading coefficients \eqref{lcP}, \eqref{lcQ}.
\end{proof}

\vspace{2mm}

Under the following assumptions on the eigenvalues of $D(t)$ and the coefficients $\kappa(\cdot,\la)$, $r(\cdot,\la)$ of the Schur complement, we give a description of the topological structure of the singular part of the essential spectrum in Theorem \ref{singstr} below.

%
%
\begin{ass}\label{ass:D.has.lim}
The eigenvalues $\la_j(t)$, $j=1,2,\dots,n$, of $D(t)$ and the coefficients $\tka(\cdot,\la)$, $r(\cdot,\la)$ of the Schur complement defined in \eqref{pi.r.varkap.def} have the following properties.
\hspace{-2mm}
\begin{enumerate}[label={{\upshape(T\arabic{*})}}]
\item \label{ass:D.has.lim.1}
The possibly improper limits $\la_{j,\beta} := \lim_{t\nearrow\beta}\la_{j}(t) \in \RR \cup \{-\infty,\infty\}$ exist and,
for some $j_0 \in \{0,1,\ldots,n\}$, 
\begin{equation}\label{ass:D.has.lim.j_0}
\lim_{t\nearrow\beta}\la_j(t)
=\begin{cases}  \la_{j,\beta}\in\RR, 
& j=1,2,\dots,j_0,
\\ 
-\infty \text{ or } \infty, & j=j_0+1,\ldots,n; 
\end{cases}
\end{equation}
moreover, there exist $m\in \mathbb{N}$ and $K\in\RR$ such that 
\begin{equation}\label{ass:max.min}
\ds\max_{j=j_0+1}^n|\la_j(t)|^2 \leq K \ds\min_{j=j_0+1}^n|\la_j(t)|^m, \quad t\in[t_\la,\beta).
\end{equation}

\item \label{ass:varkappa.holom} 
There are real constants $\phi_{\beta}$, $\psi_{\beta}$, and $\mu_{j,\beta}$, $j=1,2,\ldots,j_0$, such that 
\begin{align}\label{ass:structure.ex.1}
-\varkappa_0(\la) = - \lim_{t\nearrow\beta} \varkappa(t,\la) =
\phi_{\beta} + \psi_{\beta}\la + \ds\sum_{j=1}^{j_0}\frac{\mu_{j,\beta}}{\la_{j,\beta}-\la}. 
\end{align}

\item \label{ass:r.holom} There exists $h_{\beta}\in\RR$ such \vspace{-2mm} that 
\begin{equation}
\label{ass:r_1_is_constant}
r_1(\la) = \lim_{t\nearrow\beta} \frac{\partial}{\partial t}r(t,\la) = h_{\beta}.
\end{equation}
\end{enumerate}
\end{ass}

Notice that \eqref{ass:structure.ex.1} resembles the representation of Nevanlinna functions \eqref{rep.of.Nevanlinna.fn}, however, $\phi_{\beta}$ and $\mu_{j,\beta}$, $j=1,2,\ldots,j_0$, are assumed to be only real. Indeed, $-\varkappa_0$ may not be a Nevanlinna function as it can be seen, for instance, in Example B where $\mu_{1,1}=-\frac{1}{4}\sigma_{1,1}\leq 0$.


The following proposition describes the form of the limit function $\tpi_2(\la)$ in terms of the $j_0$ proper
limits in Assumption \ref{ass:D.has.lim.1} in Case (III) where $\pi_0(\la)=\pi_1(\la)=0$.
%
%
\begin{proposition}
\label{eshmat1}
Suppose that the assumptions of Theorem {\rm \ref{prop:suff}} \ref{prop:suff.III} are satisfied, i.e.\ Assumptions \ref{ass:regularity}, \ref{ass:suff}, \ref{ass:pres.reg} hold, Case {\rm (III)} $\pi_0(\la)=\pi_1(\la)=0$ prevails and the limits in \eqref{suff.lim.pi.r.kappa} exist. 
If $\sigma_j$, $j=1,2,\ldots,n$, are the positive function in \eqref{rep:-pi} and Assumption \ref{ass:D.has.lim.1} is satisfied with $j_0\in\{1,2,\dots,n\}$, 
then $\la\mapsto - \pi_2(\la)$ is a Nevanlinna function, i.e.\ for $\la \in \RR \setminus \big( \sessreg(\cA) \cup \Lambda_\beta(D) \big)$,
\begin{align}
\label{form.of.pi_2}
\ds-\tpi_2(\la) = - \frac 12 \lim_{t\nearrow\beta} \frac{\partial^2}{\partial t^2} \tpi(\cdot,\la) = f_{\beta} + g_{\beta}\la + \sum_{j=1}^{j_0}\frac{\sigma_{j,\beta}}{\la_{j,\beta}-\la}
\end{align}
\vspace{-1mm}
with $f_\beta\in\RR$, $g_\beta \ge 0$, and $\sigma_{j,\beta}\ge 0$, $j=1,2,\dots,j_0$. Moreover, the following limits exist and \vspace{-1mm}satisfy
\begin{align}
\label{prop:sigma.f}
&\lim_{t\nearrow\beta}\frac{\sigma_j(t)}{(\beta-t)^2} = \sigma_{j,\beta} \geq 0, \ j=1,2,\dots,j_0, 
&&\lim_{t\nearrow\beta}\frac{1}{(\beta-t)^2}\bigg(\!\!-p(t)+\!\!\!\!\sum_{j=j_0+1}^n\frac{\sigma_j(t)}{\la_j(t)}\bigg) = f_{\beta} \in \RR, \\ \label{prop:g.0}
&\lim_{t\nearrow\beta}\frac{1}{(\beta-t)^2} \sum_{j=j_0+1}^n\frac{\sigma_j(t)}{\la_j(t)^2} = g_{\beta} \geq 0, &&\lim_{t\nearrow\beta}\frac{1}{(\beta-t)^2} \sum_{j=j_0+1}^n\frac{\sigma_j(t)}{\la_j(t)^{k}} = 0, \quad k=3,4,\dots.
\end{align}
\end{proposition}
\begin{proof}
Let $\sessreg(\cA) \neq \RR$, $j\in \{j_0+1,\ldots,n\}$ be arbitrary, and $\la \in \RR \setminus \big( \sessreg(\cA) \cup \Lambda_\beta(D) \big)$. 
We may assume that $t_\la$ is chosen so large \vspace{-1mm} that 
\begin{equation}
\label{str:geom.ser}
\bigg|\frac{\la}{\la_k(t)}\bigg| \leq \frac{1}{2},  \quad t\in[t_\la,\beta), \; \ k=j_0+1, \ldots, n.
\end{equation}
Then, because $\sigma_j(t) \leq \|b(t)\|^2_{\CC^n}$, $t\in[t_\la,\beta)$, by Lemma~\ref{nice lemma1} and $\la_{k,\beta} \in \{-\infty, \infty\}$ for  $k=j_0+1,\ldots,n$, 
we obtain, for \vspace{-1mm} $t\in[t_\la,\beta)$,
\begin{align*}
\frac{\sigma_j(t)}{|\la_j(t)|^m} & \leq \frac{1}{|\la_j(t)|^m}\|D(t)-\la\|^2_{\CC^n}\|(D(t)-\la)^{-1}b(t)\|^2_{\CC^n} 
\leq \frac{1}{|\la_j(t)|^m}\ds\max_{k=j_0+1}^n|\la_{k}(t)-\la|^2\|(D(t)-\la)^{-1}b(t)\|^2_{\CC_n} \\
& \leq K_1\frac{\ds\max_{k=j_0+1}^n|\la_{k}(t)|^2}{\ds\min_{k=j_0+1}^n|\la_{k}(t)|^m}\|(D(t)-\la)^{-1}b(t)\|^2_{\CC_n},
\end{align*}
where $K_1>1$ is a constant. Therefore, \eqref{ass:max.min} and \eqref{ass:S.r} in Case (III) imply  
\begin{equation}
\label{str:sigma.lambda.m=O(1)}
\ds\frac{1}{(\beta-t)^2}\frac{\sigma_j(t)}{|\la_j(t)|^m} = \O(1), \quad t\nearrow \beta.
\end{equation}
Thus, due to \eqref{str:geom.ser}, we can \vspace{-1mm}expand
\begin{align*}
\frac{1}{(\beta-t)^2}\frac{\sigma_j(t)}{\la_j(t)-\la} =& \ds\frac{1}{(\beta-t)^2}\frac{\sigma_j(t)}{\la_j(t)}\frac{1}{1-\frac{\la}{\la_j(t)}} = \frac{1}{(\beta-t)^2}\frac{\sigma_j(t)}{\la_j(t)}\sum_{k=0}^{\infty}\Bigl(\frac{\la}{\la_j(t)}\Bigr)^k \\
=& 
\sum_{k=0}^{m-1}\frac{1}{(\beta-t)^2}\frac{\sigma_j(t)}{\la_j(t)^{k+1}} \la^k + \underbrace{\frac{\la^m}{\la_j(t)}\frac{1}{(\beta-t)^2}\frac{\sigma_j(t)}{\la_j(t)^m}\ds\sum_{\ell=0}^{\infty}\biggl(\frac{\la}{\la_j(t)}\biggr)^{\ell}}_{=\o(1), \quad  t\nearrow\beta};
\end{align*}
note that it follows from \eqref{ass:D.has.lim.j_0}, \eqref{str:sigma.lambda.m=O(1)}, and \eqref{str:geom.ser} that the last term is $\o(1)$.
Using this in \eqref{rep:-pi}, we \vspace{-1mm} find 
\begin{align}
\nonumber
-\frac{\tpi(t,\la)}{(\beta-t)^2} =& -\frac{p(t)}{(\beta-t)^2} + \sum_{j=1}^{j_0}\frac{1}{(\beta-t)^2}\frac{\sigma_j(t)}{\la_j(t)-\la} + \sum_{j=j_0+1}^n\frac{1}{(\beta-t)^2}\frac{\sigma_j(t)}{\la_j(t)-\la}\\[-1mm]
\nonumber
=&  -\frac{p(t)}{(\beta-t)^2} + \sum_{j=1}^{j_0}\frac{1}{(\beta-t)^2}\frac{\sigma_j(t)}{\la_j(t)-\la} + \sum_{j=j_0+1}^n\frac{1}{(\beta-t)^2}\frac{\sigma_j(t)}{\la_j(t)}\\[-1mm]
\label{nev-coeff}
& + \la \sum_{j=j_0+1}^n\frac{1}{(\beta-t)^2}\frac{\sigma_j(t)}{\la_j(t)^2} + \sum_{j=j_0+1}^n\sum_{k=2}^{m-1}\frac{1}{(\beta-t)^2}\frac{\sigma_j(t)}{\la_j(t)^{k+1}}\la^k + \o(1) \\[-1mm]
\nonumber
=& \frac{1}{(\beta-t)^2}\bigg(-p(t)+\sum_{j=j_0+1}^n\frac{\sigma_j(t)}{\la_j(t)}\bigg) + \la \frac{1}{(\beta-t)^2}\sum_{j=j_0+1}^n\frac{\sigma_j(t)}{\la_j(t)^2}\\[-1mm]
\nonumber
&+ \sum_{j=1}^{j_0}\frac{\sigma_j(t)}{(\beta-t)^2}\frac{1}{\la_j(t)-\la} + \sum_{k=2}^{m-1}\la^k\bigg(\frac{1}{(\beta-t)^2} \sum_{j=j_0+1}^n\frac{\sigma_j(t)}{\la_j(t)^{k+1}}\bigg) + \o(1),  \quad t\nearrow\beta.                         
\end{align}
By Theorem \ref{prop:suff} \ref{prop:suff.III}, $\frac{\tpi(t,\la)}{(\beta-t)^2}$ has a limit as $t\nearrow\beta$ which coincides with the limit $\tpi_2(\la)$ in \eqref{suff.lim.pi.r.kappa}. 
Since $\sessreg(\cA)\ne \RR$ is closed, \eqref{nev-coeff} holds for infinitely many $\la \!\in\! \RR \!\setminus\! \big( \sessreg(\cA) \cup \Lambda_\beta(D) \big)$ and so the existence of the limits \eqref{prop:sigma.f}, \eqref{prop:g.0} follows from \eqref{ass:D.has.lim.j_0} and \eqref{nev-coeff}. Since $\sigma_j(t) > 0$, $t\in[\alpha,\beta)$, the properties of the limits in \eqref{prop:sigma.f}, \eqref{prop:g.0} are obvious except for the second one in \eqref{prop:g.0}; the latter follows from \eqref{ass:D.has.lim.j_0} and the existence of the first limit \vspace{-1mm} in~\eqref{prop:g.0}~as
\begin{align*}
\label{ineq.nevanlinna.last}
\bigg|\frac{1}{(\beta-t)^2} \sum_{j=j_0+1}^n\frac{\sigma_j(t)}{\la_j(t)^{k+1}}\bigg| \leq \ds \frac 1{\min_{j=j_0+1}^n|\la_j(t)|^{k-1}} \!\!\sum_{j=j_0+1}^n\frac{1}{(\beta-t)^2}\frac{\sigma_j(t)}{\la_j(t)^2}
\to 0, \quad t\nearrow\beta, 
\quad k=2,3,\ldots,m-1.
\end{align*}
Hence $-\tpi_2(\la)$ satisfies \eqref{form.of.pi_2} and it is a Nevanlinna function since $g_{\beta}$, $\sigma_{j,\beta}\geq 0$, $j=1,2,\dots,j_0$.
\end{proof}
Theorem \ref{prop:suff} \ref{prop:suff.III} and Proposition \ref{eshmat1} together yield the following result on the topological structure of the singular part of the essential spectrum outside of the limiting set $\Lambda_\beta(D) = \{\la_{1,\beta}, \la_{2,\beta}, \ldots, \la_{j_0,\beta} \}$.

Although the existence of the limits $\tpi_2(\la)$, $\varkappa_0(\la)$, and $r_1(\la)$ is guaranteed by our assumptions only for $\la \in \RR \setminus \big( \sessreg(\cA) \cup \Lambda_\beta(D) \big)$, the rational functions representing these limits in \eqref{form.of.pi_2}, \eqref{ass:structure.ex.1}, and \eqref{ass:r_1_is_constant} are defined for all $\lambda\in\RR\setminus\Lambda_{\beta}(D)$; this observation is used in the following theorem. 

\begin{theorem}
\label{singstr}
Let $\beta< \infty$.
Assume that the assumptions of Theorem {\rm \ref{prop:suff}} \ref{prop:suff.III} are satisfied, 
i.e.\ Assumptions \ref{ass:regularity}, \ref{ass:suff}, \ref{ass:pres.reg} hold, Case {\rm (III)} $\pi_0(\la)=\pi_1(\la)=0$ prevails and the limits in \eqref{suff.lim.pi.r.kappa} exist.  
Further, suppose that Assumption \ref{ass:D.has.lim} holds and the coefficients $f_\beta$, $g_\beta$, $\sigma_{j,\beta}$ in \eqref{form.of.pi_2}, $\psi_\beta$ in \eqref{ass:structure.ex.1}, and $h_\beta$ in \eqref{ass:r_1_is_constant} satisfy
$g_{\beta}+4\psi_{\beta}\neq0$ and $h_{\beta}^2+\psi_{\beta}\sum_{j=1}^{j_0}\sigma_{j,\beta} \neq 0$. If we extend the functions $\tpi_2$, $\varkappa_0$, and $r_1$ by means of \eqref{form.of.pi_2}, \eqref{ass:structure.ex.1}, and \eqref{ass:r_1_is_constant}, respectively, to $\RR\setminus\Lambda_{\beta}(D)$ and define the \vspace{-0.9mm} set
\[ 
\Sigma := \Big\{\la\in\RR\setminus \Lambda_{\beta}(D):\, r_1(\la)^2-\varkappa_0(\la)\tpi_2(\la)\geq\frac{1}{4}\tpi_2(\la)^2\Big\}, 
\]
then $\ov{\Sigma} \subset \sess(\cA)$ \vspace{-2mm} and
\begin{align}
\label{sess.sing.str}
\sesssing(\cA)\setminus\Lambda_{\beta}(D) = \ov{\Sigma}\setminus\bigl(\sessreg(\cA)\cup\Lambda_{\beta}(D)\bigr).
\end{align}
Moreover, the set $\ov{\Sigma}$ has the following structure in terms of the coefficient functions in \eqref{ass:structure.ex.1}--\eqref{form.of.pi_2}.

\begin{enumerate}[label={\rm(\alph*)}]

\item If $\,g_\beta>0$, $\ov{\Sigma}$ 
consists of \vspace{0.5mm}

- the union of at most $j_0+1$ compact intervals if $\,g_{\beta} + 4\psi_{\beta}>0$; \vspace{0.5mm}

- the union of $(-\infty, s_1]$, at most $j_0$ compact intervals and $[s_2,\infty)$ if $\,g_{\beta} + 4\psi_{\beta}<0$. \vspace{0.5mm} 

\item If $\,g_{\beta}=0$ and $f_{\beta}\neq 0$, $\ov{\Sigma}$ 
consists ~of \vspace{0.5mm} 

- the union of $(-\infty, s]$ and at most $j_0$ compact intervals if $\,f_{\beta}\psi_{\beta}>0$;\vspace{0.5mm} 

- the union of at most $j_0$ compact intervals and $[s,\infty)$ if $\,f_{\beta}\psi_{\beta}<0$.\vspace{0.5mm} 

\item If $\,g_{\beta} = f_{\beta} = 0$, $\ov{\Sigma}$ 
consists of \vspace{0.5mm} 

- the union of at most $j_0$ compact intervals if $\,h_{\beta}^2 + \psi_{\beta}\sum_{j=1}^{j_0}\sigma_{j,\beta} < 0$;

- the union of $(-\infty, s_1]$, at most $j_0\!-\!1$ compact intervals and $[s_2,\infty)$ if $\,h_{\beta}^2 \!+ \psi_{\beta}\sum_{j=1}^{j_0}\sigma_{j,\beta} > 0$. 
\end{enumerate}
\end{theorem}
\begin{proof}
By the definition of $\Sigma$ and \eqref{sing_part_suff}, we \vspace{-2mm} have 
\begin{align}
\label{sesssing=Sigma!}
\Sigma\setminus\sessreg(\cA) = \sesssing(\cA)\setminus\Lambda_{\beta}(D),
\vspace*{-4mm} 
\end{align}
and \vspace{-1.5mm} hence 
\begin{align*}
\Sigma \subset (\Sigma\setminus\sessreg(\cA))\cup\sessreg(\cA) \subset \sesssing(\cA)\cup\sessreg(\cA)=\sess(\cA).
\vspace{-2mm} 
\end{align*}
Since the essential spectrum is closed, it follows that $\ov{\Sigma}\subset\sess(\cA)$.

To prove \eqref{sess.sing.str}, we first note that there are polynomials $P_{\sigma}$, $P_{\mu}$ of degree at most $j_0-2$ in $\la$ such that
\begin{align*}
-\tpi_2(\la) = f_{\beta}+g_{\beta}\la + \frac{K_{\sigma}\la^{j_0-1}+P_{\sigma}(\la)}{\prod_{j=1}^{j_0}(\la_{j,\beta}-\la)}, \quad -\varkappa_0(\la) = \phi_{\beta}+\psi_{\beta}\la + \frac{K_{\mu}\la^{j_0-1}+P_{\mu}(\la)}{\prod_{j=1}^{j_0}(\la_{j,\beta}-\la)}, 
\end{align*}
with $K_{\sigma} = (-1)^{j_0-1}\sum_{j=1}^{j_0}\sigma_{j,\beta}$, $K_{\mu} = (-1)^{j_0-1}\sum_{j=1}^{j_0}\mu_{j,\beta}$. Hence $\Sigma$ is the union of a finite number of intervals and the endpoints of these intervals that are not in $\Sigma$ belong to $\Lambda_{\beta}(D)$, \vspace{-1mm} i.e. 
\begin{align}
\label{imp.Sigma}
\ov{\Sigma}\setminus\Sigma\subset\Lambda_{\beta}(D).
\vspace{-1mm} 
\end{align}
This implies  \vspace{-1mm}  that
\begin{align}
\label{eq:ov.Sigma1}
\ov{\Sigma}\setminus\bigl(\sessreg(\cA)\cup\Lambda_{\beta}(D)\bigr) = \Sigma\setminus\bigl(\sessreg(\cA)\cup\Lambda_{\beta}(D)\bigr) = \Sigma\setminus \sessreg(\cA).
\vspace{-1mm} 
\end{align}
Now \eqref{sess.sing.str} follows from \eqref{sesssing=Sigma!} and \eqref{eq:ov.Sigma1}.

To prove the claims in (a), (b), (c), observe that there are polynomials $P_3$, $P_4$, $P_5$ with $\deg P_3\leq 2j_0+1$, $\deg P_4\leq 2j_0$, $\deg P_5\leq 2j_0-1$, respectively, such \vspace{1mm} that   
\begin{align}
\label{charac.str.thm1}
&\Sigma=\Big\{\la\in\RR \setminus \Lambda_{\beta}(D): \, g_{\beta}(g_{\beta}+4\psi_{\beta})\la^{2j_0+2}+P_3(\la) \leq 0 \Big\}  &&\text{if} \quad g_{\beta}>0,
\\
\label{charac.str.thm2}
&\Sigma=\Big\{\la\in\RR \setminus \Lambda_{\beta}(D): \, f_{\beta}\psi_{\beta}\la^{2j_0+1}+P_4(\la) \leq 0 \Big\}  &&\text{if} \quad g_{\beta}=0,\; f_{\beta}\neq 0,
\\[-2mm]
\label{charac.str.thm3}
&\Sigma=\Big\{\la\in\RR \setminus \Lambda_{\beta}(D): \, \Big(h_{\beta}^2 + \psi_{\beta}\sum_{j=1}^{j_0}\sigma_{j,\beta}\Big) \la^{2j_0} + P_5(\la) \geq 0 \Big\}  &&\text{if} \quad g_{\beta}=0,\; f_{\beta}= 0.
\end{align}
Now the claims follow from \eqref{charac.str.thm2} by elementary sign considerations using the various assumptions on $g_{\beta}$, $\psi_{\beta}$, $h_{\beta}$, and $\sigma_{j,\beta}$ in (a), (b), and~(c). 
\end{proof}
Figure \ref{fig:str} shows a possible location of the sets in Proposition \ref{ineq:reg.1} and Theorem \ref{singstr} with $n\geq4$, $j_0=4$ in the second case of (a) or (c). 
\begin{figure}[htb!]
\includegraphics[width=0.85 \textwidth]{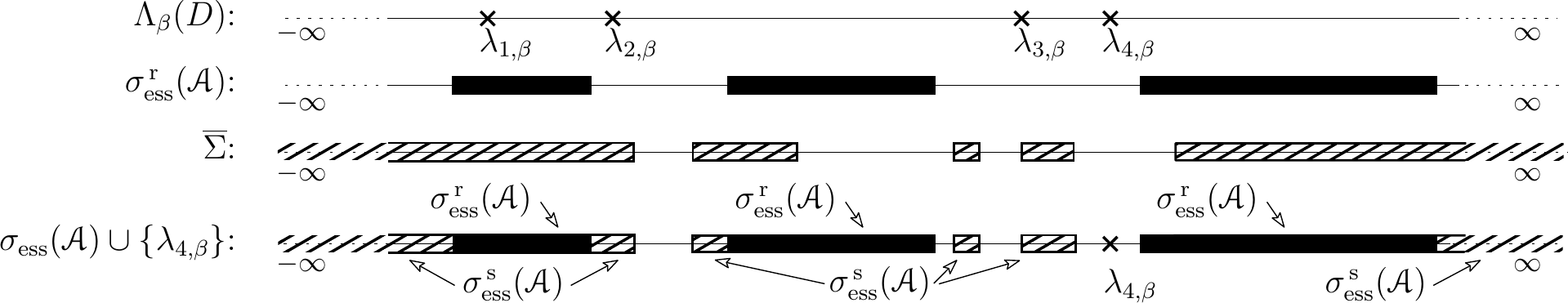}
\caption{Illustration of a possible structure of the essential spectrum of $\cA$.}
\label{fig:str}
\end{figure}

Note that the regular part or the singular part of the essential spectrum may contain some points of the exceptional set $\Lambda_{\beta}(D)$, see Example B.
In Figure \ref{fig:str}, nothing can be said about the point $\la_{4,\beta}$ while the closedness of the essential spectrum yields that $\la_{i,\beta} \in \sess(\cA)$ for $i=1,2,3$.

\begin{remark}
\label{rem:ess.spec.radius}
We define the \emph{essential spectral radius}, introduced in \cite{Nussbaum70} for bounded linear operators, of $\cA$~as
\[
r_{\rm{ess}}(\cA) := \sup\{|\la|: \, \la\in\sess(\cA)\} \in [0,\infty].
\]
Then Theorem \ref{singstr} shows that $r_{\rm{ess}}(\cA)<\infty$  if and only if the regular part of the essential spectrum is bounded and 
for the singular part the first case of either (a) or (c) prevails. 
Moreover, a bound for $r_{\rm{ess}}(\cA)$ can be given by estimating the roots with largest absolute value of the polynomials in \eqref{charac.str.thm1}, \eqref{charac.str.thm3}. In particular, if $D$ is scalar i.e.\  $n=1$ and $j_0=0$, then $r_{\rm{ess}}(\cA)<\infty$ only if we are in the first case of (a); in this case $\Sigma=[\la^-,\la^+]$ and thus
\[
r_{\rm{ess}}(\cA) = \max \Big\{ \sup_{t\in[\alpha,\beta)} |\Delta(t)|, |\la^+|, |\la^-| \Big\}
\vspace{-4mm}
\]
where 
\begin{align*}
\la^\pm \!:=\! \frac{-N \pm \sqrt{N^2\!-\!MK}}{M}, \quad
M:=g_\beta(g_\beta+4\psi_\beta), \quad 
N:=f_\beta(g_\beta+2\psi_\beta)+2g_\beta\phi_\beta, \quad 
K:=f_\beta(f_\beta+4\phi_\beta)-4h_\beta^2.
\end{align*}
\end{remark}

%
\section{Systematic analysis of typical examples}

In this section, we apply our results {to two different typical examples of which one arises in linear magnetohydrodynamic stability analysis.
We identify the particular assumptions under which these examples were treated in earlier papers with special cases of our general classification, i.e.\ Cases (I), (II), (III) in Section \ref{subsec:suff}, and, using our abstract results, we provide a complete analysis of the essential spectrum of these examples in all cases.

In Example A, we show that the so-called quasi-regularity conditions in \cite{KN03} and \cite{QCH11} mean that Case (III) prevails, the singular part of the essential spectrum is non-empty and can be computed by our abstract results (see our earlier work \cite{ILLT13}).
The paper \cite{KLN08} where the quasi-regularity conditions are not satisfied is an example for Cases (I) or~(II);  here our abstract Theorem~\ref{prop:suff} yields that the singular part of the essential spectrum is always~empty.

Example B is a more general model of an operator arising in linearised magnetohydrodynamics (MHD) which describes the oscillations of plasma in an equilibrium configuration in a cylindrical domain and whose essential spectrum was first calculated by Kako \cite{Kako87}.
The quasi-regularity conditions assumed in the papers \cite{HMN99}, \cite{MNT02}, and 
\cite{FMM00} amount to Case (III) and we compute the singular part of the essential spectrum by means of our abstract result Theorem \ref{thm:ess.pres}. We mention that \cite{FMM00} also contains results for the non-symmetric case.
It seems that \cite{Moeller04} is the only paper where Case (III) does not prevail; its assumptions correspond to the first case of Case (I). 
Here we discuss all possible cases, in particular, the second case of Case (I) and Case (II)  not covered by any earlier~work.

In both examples, we adopt to the following strategy: if necessary, we transform the singularity to the right endpoint of the interval by means of a unitary transformation. Next, we determine the exceptional set $\Lambda_{\beta}(D)$. Then we show that Assumption \ref{ass:suff} in Cases (I), (III) and 
Assumption~\ref{ass:suff'} in Case (II) are satisfied. Finally, we verify all requirements of Theorem~\ref{thm:ess.pres} or of Theorem~\ref{prop:suff} resp.\ Remark \ref{rem:SII} and use them to describe the essential spectrum.

\subsection*{Example A}
Let $\wt\varrho, \wt m \in C^1([0,1],\RR)$, $\wt\psi \in C^1([0,1],\CC)$, $\wt\phi \in C([0,1],\RR)$. Assume $\wt\varrho > 0$ on $[0,1]$ and $\wt m(0)\ne0$. 
Consider the operator \vspace{-2mm} matrix
\begin{align}
\label{KLNO8A_0}
\wt\cA_0 = \begin{pmatrix} -\dx \wt\varrho \dx + \wt\phi & \ds\dx \frac{\wt\psi}{x} \\[1.8ex]
\ds-\frac{\ov{\wt\psi}}{x} \dx & \ds\frac{\wt m}{x^2} \end{pmatrix}
\end{align}
with domain $\cD(\wt\cA_0)=C_0^{2}((0,1)) \oplus C_0^{1}((0,1))$ in the Hilbert space $L^2((0,1)) \oplus L^2((0,1))$
and let $\wt\cA$ be an arbitrary closed symmetric extension of $\wt\cA_0$. 

\subsubsection*{Transformation to the form~\eqref{A02} and verification of Assumption~\ref{ass:regularity}}\!
If we introduce the unitary \vspace{-0.1mm} trans\-formations
\begin{equation*}
U:L^2((0,1)) \to L^2((0,1)), \quad (Uu)(x):=u(1-x), 
\end{equation*}
and $\mathcal{U}:=\text{diag}(U,U)$, then $\cA_0 := \mathcal{U}\wt\cA_0\mathcal{U}^{-1}$ has the form \eqref{A02} with $n=1$ and coefficients  
\begin{align*}
p(t) := \varrho(t), \quad q(t): =\phi(t), \quad b(t) := \frac{\ov\psi(t)}{1-t}, \quad c(t) := 0, \quad D(t) := \frac{m(t)}{(1-t)^2}, 
\end{align*}
for $t\in[0,1)$ where $\varrho(t) := \wt\varrho(1-t)$, $\psi(t) := \wt\psi(1-t)$, $m(t) := \wt m(1-t)$ and $\phi(t) := \wt\phi(1-t)$. By the smoothness assumptions on the coefficients of $\wt\cA_0$, the coefficients of $\cA_0$ satisfy Assumption~\ref{ass:regularity}.

\subsubsection*{The set $\Lambda_\beta(D)$}
Since $\wt m(0)\!\neq\!0$, we have $\lim\limits_{t\nearrow 1} D(t)\!=\! \lim\limits_{t\nearrow 1}\dfrac{m(t)}{(1-t)^2} \!\in\! \{ - \infty, \infty \}$, hence, by  \eqref{Lam.beta.lim}, 
\[
\Lambda_\beta(D) \!=\! \emptyset.
\]
 
\subsubsection*{Verification of Assumption \ref{ass:suff} resp.\ \ref{ass:suff'}}

Elementary calculations show that the coefficients in the asymptotic expansion of $\tpi(t,\la)$ in \eqref{pi_asymp} are given by
\begin{align*}
\tpi_0(\la) = \frac{(\varrho m - |\psi|^2)(1)}{m(1)}, \quad \tpi_1 (\la)= -\frac{m'(1)}{m(1)}\tpi_0(\la) + \frac{(\varrho m - |\psi|^2)'(1)}{m(1)}. 
\end{align*}
Therefore, in terms of the original coefficients in \eqref{KLNO8A_0}, the three cases in \eqref{Cases} can be classified as 
\begin{align}
\nonumber
&\textbf{Case (I)}:  &&(\wt\varrho \wt m-|\wt\psi|^2)(0) \neq 0; &\hspace{888mm}
\\\nonumber
&\textbf{Case (II)}: &&(\wt\varrho \wt m-|\wt\psi|^2)(0)=0,\ \ (\wt\varrho \wt m-|\wt\psi|^2)'(0) \neq 0;&
\\\nonumber
&\textbf{Case (III)}:&&(\wt\varrho \wt m-|\wt\psi|^2)(0)=0,\ \ (\wt\varrho \wt m-|\wt\psi|^2)'(0) = 0.&
\end{align}
To verify Assumption \ref{ass:suff}, note that $c\equiv0$, $r \equiv 0$, $\varkappa(\cdot,\la) = \phi-\la \in C^1([0,1])$, and
\[
\lim_{t\nearrow 1}\frac{1}{1-t}\frac{b(t)}{D(t)-\la} = \frac{\wt\psi(0)}{\wt m(0)}.
\]
Hence all conditions \eqref{ass:S.Rb}--\eqref{ass:S.kappac} in Assumption \ref{ass:suff} for Case (I), and  all conditions  \eqref{rem:SII.Rb}--\eqref{rem:SII.kappac} in Assumption \ref{ass:suff'} for Case (II) are satisfied (cf.\ Remark \ref{rem:SII}).

\subsubsection*{Essential spectrum in Cases {\upshape(I)} and {\upshape(II)}}

By what was shown above, we can apply Theorem \ref{prop:suff} \ref{prop:suff.I} which yields that, in Cases (I) and (II), 
\[
\sesssing\bigl(\wt\cA\bigr) = \sesssing(\cA) = \emptyset.
\]
Therefore it follows from Theorem \ref{thm:ess.abs} that 
$$\sess(\wt\cA) = \sess(\cA) = \sessreg(\cA) = \ov{\Delta([0,1))} = \ov{\wt\Delta((0,1])},$$
\vspace{-1mm}where  
\[
\Delta(t)=\wt\Delta(1-t), \quad t\in [0,1), \qquad \wt\Delta(t) = \frac 1{t^2} \Big(\wt m(t) - \frac{|\wt\psi(t)|^2}{\wt\varrho (t)}\Big), \quad t\in (0,1]. 
\]

\subsubsection*{Essential spectrum in Case {\upshape(III)}}
This case was already treated in \cite[Example 7.2]{ILLT13} since it satisfied the stronger assumptions therein, and we just include the result for completeness. Here $\sesssing(\wt\cA) \ne \emptyset$ and
\begin{equation}
\label{verylast}
  \sess(\wt\cA) = \sessreg(\wt\cA) \cup \sesssing(\wt\cA) = \ov{\wt\Delta((0,1])} \cup {\rm conv}\, \Big\{ \wt\Delta_0 , \frac{4\wt m(0)\wt\phi(0)+\wt\rho(0)\wt\Delta_0}{4\wt m(0)+\wt\rho(0)} \Big\}
\end{equation}
where conv denotes the convex hull and  \vspace{-2mm} $\wt\Delta_0 := \lim_{t\searrow 0} \wt \Delta(t)$.

\subsubsection*{Essential spectral radius}
In Cases (I) and (II), we have $r_{\rm{ess}}(\wt\cA)= \infty$ as $\wt\Delta$ is not bounded. In Case (III), $r_{\rm{ess}}(\wt\cA)<\infty$ and we conclude from \eqref{verylast}, observing $\wt\Delta_0 \in \ov{\wt\Delta((0,1])}$,
\[
r_{\rm{ess}}(\wt\cA)=\max\bigg\{ \sup_{t\in(0,1]} \big| \wt \Delta(t) \Big|     
, \, \Bigl|\frac{4\wt m(0)\wt\phi(0)+\wt\rho(0)\wt\Delta_0}{4\wt m(0)+\wt\rho(0)}\Bigr|\bigg\}.
\]

\subsection*{Example B}
\label{moeller_example}
Let $\wt\vartheta,\wt\delta_{11},\wt\delta_{22},\wt\beta_1, \wt\beta_2 \in C^1([0,1],\RR)$, $\wt\gamma,\wt\delta_{12}\in C^1([0,1],\CC)$, $\wt\phi\in C([0,1],\RR)$ and assume that $\wt\vartheta(x)\ne0$, $x\in[0,1]$, and $\wt\delta_{11}(0)\ne0$. We consider the operator matrix
\begin{align}
\label{3x3example}
\wt\cA_0=\left(\begin{array}{c|cc}
\ds-\dx\frac{\wt\vartheta}{x}\dx x+\wt\phi
& \ds\ii\dx\frac{\wt\beta_1}{x}+\wt\gamma & \ii\dx\wt\beta_2 \\[1.5ex]
\hline
\ds\rule{0ex}{4.5ex}\frac{\wt\beta_1}{x^2}\ii\dx x+\ov{\wt\gamma}
& \ds\frac{\wt\delta_{11}}{x^2} & \ds\frac{\wt\delta_{12}}{x} \\[1.5ex]
\ds\frac{\wt\beta_2}{x}\ii\dx x
& \ds\frac{\,\ov{\wt\delta_{12}}\,}{x} & \wt\delta_{22}
\end{array}\right)
\end{align}
with domain $\cD(\wt\cA_0) \!=\! C^2_0(0,1)\oplus(C^1_0(0,1))^2$ in the 
direct 
sum $L^2((0,1),x) \oplus (L^2((0,1),x))^2$ of weighted $L^2$-spaces. 
\subsubsection*{Transformation to the form \eqref{A02} and verification of Assumption \ref{ass:regularity}}
If we introduce the unitary \vspace{-0.1mm} trans\-formations
\[
U:L^2((0,1),x) \to L^2((0,1)), \quad (Uu)(x) := \sqrt{1-x}\,u(1-x),
\]
and $\mathcal{U}:=\text{diag}(U,U)$, then $\cA_0 := \mathcal{U}\wt\cA_0\mathcal{U}^{-1}$ has the form \eqref{A02} 
with $n=2$ and coefficients 
\begin{equation}\label{ex.b.coef}
\hspace{-10mm}
\begin{array}{ll}
p(t) \!:=\! \vartheta(t), & \ds q(t)  \! :=  \! \phi(t) \!+\! 
\frac{1}{2}\frac{\partial}{\partial 
	t}\frac{\vartheta(t)}{1-t}\!+\!\frac{1}{4}\frac{\vartheta(t)}{(1-t)^2},\\
\ds
b(t) \!:=\! -\ii \left(\frac{\beta_1(t)}{1-t}, \, \beta_2(t) \right)^{\!\! \mathrm{t}}\!\!,
&
\ds
c(t) \! :=\! -\frac{b(t)}{2(1-t)}+
\begin{pmatrix}
\,\ov{\gamma(t)}\, \\ 0
\end{pmatrix},
\end{array}
\quad
D(t)\!:=\!
\begin{pmatrix}
\ds \frac{\delta_{11}(t)}{(1-t)^2} \!&\!\ds \frac{\delta_{12}(t)}{1-t}\\[2ex]
\ds \frac{\ov{\delta_{12}(t)}}{1-t} \!&\! \ds\delta_{22}(t)
\end{pmatrix}\!\!,
\hspace*{-4.5mm} 
\end{equation} 
for $t\!\in\! [0,1)$ with $\vartheta(t)\!:=\!\wt\vartheta(1\!-\!t)$, $\phi(t)\!:=\!\wt\phi(1\!-\!t)$, 
$\gamma(t)\!:=\!\wt\gamma(1\!-\!t)$, $\beta_i(t)\!:=\!\wt\beta_i(1\!-\!t)$, and 
$\delta_{ij}(t)\!:=\!\wt\delta_{ij}(1\!-\!t)$, $i,j\!=\!1,2$. 
By the smoothness assumptions on the coefficients of $\wt\cA_0$, the coefficients of $\cA_0$ satisfy Assumption~\ref{ass:regularity}.

The following functions related to the Hermitian matrix-valued function $\Delta$ given by
\begin{align}
\label{moeller:Delta-tilde}
\Delta(t)\!=\!\wt\Delta(1-t), \ t\!\in\! [0,1), \quad 
\wt\Delta(t) \!=\! \begin{pmatrix}
\ds\frac1{t^2}\biggl(\wt\delta_{11}-\frac{\wt\beta_1^2}{\wt\vartheta}\biggr)(t)
& \ds\frac1{t}\biggl(\wt\delta_{12}-\frac{\wt\beta_1\wt\beta_2}{\wt\vartheta}\biggr)(t) \\[2.3ex]
\ds\frac1{t}\biggl(\ov{\wt\delta_{12}}-\frac{\wt\beta_1\wt\beta_2}{\wt\vartheta}\biggr)(t)
& \ds\biggl(\wt\delta_{22}-\frac{{\wt\beta_2}^2}{\wt\vartheta}\biggr)(t)
\end{pmatrix}, \quad t\!\in\! (0,1],
\end{align}
play an important role in the sequel: 
\begin{alignat}{2}
\Psi(t) :=& (1-t)^2 \vartheta(t) \ \text{tr}(\Delta(t)) = (\vartheta \delta_{11}-\beta_1^2)(t)+(1-t)^2(\vartheta \delta_{22}-\beta_2^2)(t), 
\quad && t\in[0,1], \label{Psi} \\
\Phi(t) :=& (1-t)^2 \vartheta(t)\det(\Delta(t)) = 
\frac{1}{\vartheta(t)}\bigl[(\vartheta \delta_{11}-\beta_1^2)(t)(\vartheta \delta_{22}-\beta_2^2)(t)-|(\vartheta \delta_{12}-\beta_1\beta_2)(t)|^2\bigr],
\quad && t\in[0,1]. \label{Phi}
\end{alignat}

\subsubsection*{The set $\Lambda_\beta(D)$ and Assumption {\rm (T1)}}
It is not difficult to check that the 
two eigenvalues $\la_1(t)$ and $\la_2(t)$ of $D(t)$, $t\in[0,1)$, can be numbered such that $\la_1(t)$ has a finite limit and $|\la_2(t)|$ tends to $\infty$ for $t\nearrow 1$; hence Assumption (T1) holds with $j_0=1$. Further, $\la_2(t)$ has the asymptotic behaviour
\begin{align}
\label{asymp:la_1}
\la_{2}(t) = \frac{\delta_{11}(1)}{(1-t)^2} + \o\Bigl(\frac{1}{(1-t)^2}\Bigr), \quad t\nearrow 1.
\end{align}
Since $\delta_{11}(1) \!=\! \wt \delta_{11}(0)\ne0$, we have $\lim\limits_{t\nearrow 1}\la_{2}(t) \!\in\! \{- \infty,\infty\}$, while for $\lambda_1(t)$ 
the limit is  \vspace{-2mm} finite, 
\begin{align}
\label{la_2,2}
\la_{1,1} = \lim_{t\nearrow 1}\la_{1}(t) = \lim_{t\nearrow 1}\frac{\det D(t)}{\la_1(t)}
=\wt\delta_{22}(0)-\frac{|\wt\delta_{12}(0)|^2}{\wt\delta_{11}(0)}.
\vspace{-1mm}
\end{align}
Hence, according to \eqref{Lam.beta.lim}, we obtain 
\[
\Lambda_\beta(D) = \{\la_{1,1}\}.
\]

\subsubsection*{Verification of Assumption \ref{ass:suff} resp.\ \ref{ass:suff'}}

In the following proposition, we compute the first two coefficients $\tpi_0(\la)$ and $\tpi_1(\la)$ of the asymptotic expansion of
$\tpi(t,\la)$ in \eqref{pi_asymp} and characterize the possible cases in~\eqref{Cases}.

\begin{proposition}
\label{lem:pi.moeller}
Let $\la \in \RR \setminus(\sessreg(\cA) \cup \Lambda_\beta(D))$. Then
$\tpi(t,\la) = \tpi_0(\la) + \tpi_1(\la)(t-1) + \o(1-t)$, $t\nearrow 1$, 
with
	\begin{align}
	\label{tpi0}
	\tpi_0(\la) & = \frac{\Phi(1)-\Psi(1)\la}{\bigl(\delta_{11}\delta_{22}-|\delta_{12}|^2-\la\delta_{11}\bigr)(1)},
	\\
	\label{tpi1}
	\tpi_1(\la) &= \frac{\Phi'(1)-\Psi'(1)\la - \tpi_0(\la)\bigl(\delta_{11}\delta_{22}-|\delta_{12}|^2-\la\delta_{11}\bigr)'(1)}{\bigl(\delta_{11}\delta_{22}-|\delta_{12}|^2-\la\delta_{11}\bigr)(1)},
	\end{align}
where functions $\Phi$ and $\Psi$ are defined in \eqref{Psi} and \eqref{Phi}, respectively. Moreover, in terms of the original coefficients in \eqref{3x3example}, the three cases in \eqref{Cases} can be classified~as
\begin{align}
	&{\text{ \bf Case (I)}}:  && (\wt\vartheta \,\wt\delta_{11}-\wt\beta_1^2)(0)\neq 0 \; \text{or} & 
	\label{Cases_moellera} 
	\\ 
	& && (\wt\vartheta \, \wt\delta_{11}-\wt\beta_1^2)(0) = 0 \; \text{and} \; (\wt\vartheta \, \wt\delta_{12}-\wt\beta_1\wt\beta_2)(0)\neq 0;& 
	\nonumber
	\\ 
	&{\text{ \bf Case (II)}}: &&(\wt\vartheta \, \wt\delta_{11}-\wt\beta_1^2)(0) = 0, \ \ (\wt\vartheta \, \wt\delta_{12}-\wt\beta_1\wt\beta_2)(0) = 0 \; \text{and} \; (\wt\vartheta \, \wt\delta_{11}-\wt\beta_1^2)'(0)\neq 0;&
	\label{Cases_moeller}
	\\
	&{\text{ \bf Case (III)}}:&&(\wt\vartheta \, \wt\delta_{11}-\wt\beta_1^2)(0) = 0, \ \ (\wt\vartheta \, \wt\delta_{12}-\wt\beta_1\wt\beta_2)(0) = 0 \; 
	\text{and} \; (\wt\vartheta \, \wt\delta_{11}-\wt\beta_1^2)'(0) = 0. &
	\label{Cases_moellerc}
	\end{align}
\end{proposition}

\begin{proof}
  Let $\la\notin \sessreg(\cA)\cup \Lambda_{\beta}(D) = \ov{\Lambda_{[0,1)}(\Delta)}\cup \Lambda_{\beta}(D)$, where $\Delta$ is as in \eqref{moeller:Delta-tilde}.
	First we note that Lemma \ref{lem:pidet} implies  that, \vspace{-2mm} for $t\in[0,1)$,
	\begin{align}\label{pi_moeller}
	\tpi (t,\la)=
	\frac{\Theta(t,\la)}{\Xi(t,\la)},
	\vspace{-2mm}
	\end{align}
	where 
	\begin{align}
	\label{pi_moellera}
	\Theta(t,\la) := & (1-t)^2 \vartheta(t) \det(\Delta(t)-\la) =
	\vartheta(t)(1-t)^2\la^2-\Psi(t)\la+\Phi(t),\\
	\label{pi_moellerb}
	\ds\Xi(t,\la) :=& (1-t)^2 \det(D(t)-\la) = \delta_{11}(t)\delta_{22}(t)-|\delta_{12}(t)|^2-\la\bigl(\delta_{11}(t)+(1-t)^2 \delta_{22}(t)\bigr)+\la^2(1-t)^2.
	\end{align}
	It is not difficult to see that, as $t\nearrow 1$, the functions $\Theta(\cdot,\la)$ and $\Xi(\cdot ,\la)$ have the asymptotic expansions
	\begin{align}
	\label{Theta}
	\Theta(t,\la) &= \Phi(1)-\Psi(1)\la + \bigl(\Phi'(1)-\Psi'(1)\la\bigr)(t-1) + \o(1-t), 
	\\
	\label{Xi}
	\Xi(t,\la) &= \bigl(\delta_{11}\delta_{22}-|\delta_{12}|^2-\la\delta_{11}\bigr)(1) + \bigl(\delta_{11}\delta_{22}-|\delta_{12}|^2-\la\delta_{11}\bigr)'(1)(t-1) + \o(1-t).
	\end{align}
	Now comparing coefficients for the powers $(1-t)^k$, $k\in\N_0$, yields that the coefficients of $\tpi(t,\la)$ in \eqref{pi_asymp}
	are given by \eqref{tpi0}, \eqref{tpi1};
	note that $\bigl(\delta_{11}\delta_{22}-|\delta_{12}|^2-\la\delta_{11}\bigr)(1) \neq 0$ since $\la\notin\Lambda_\beta(D)$. 
	In order to prove the characterization of Cases (I), (II), (III), by \eqref{tpi0}, \eqref{tpi1}, and \eqref{Psi}, \eqref{Phi}, 
	it suffices to show that
	\begin{align}
	\label{tpi0=0}
	\tpi_0(\la) = 0 \ & \iff \ \Psi(1)=0, \ 
	\Phi(1)= 0\\[-3mm]
	\intertext{and, if \vspace{-2mm} $\tpi_0(\la) = 0$,}
	\label{tpi1=0}
	\tpi_1(\la) = 0 \ & \iff \ \Psi'(1)=0.
	\end{align}
	To this end, we first observe that, by \eqref{pi_moellera}, 
	\[
	\det(\Delta(t)-\la) = \frac 1 {\vartheta(\la)} \frac 1{(1-t)^2} \Theta(t,\la) =
	\la^2-\frac{\Psi(t)}{\vartheta(t)}\frac{1}{(1-t)^2}\la + \frac{\Phi(t)}{\vartheta(t)}\frac{1}{(1-t)^2},
	\quad t\in [0,1),
	\]
	and thus the eigenvalues $\la_{\pm}(\Delta(t))$ of $\Delta(t)$ are given by
	\begin{equation}
	\label{lapmDelta}
	\la_{\pm}(\Delta(t)) = \frac{\Psi(t) \pm \sqrt{\Psi(t)^2-4\vartheta(t)\Phi(t)(1-t)^2}}{2\vartheta(t)(1-t)^2},
	\quad t\in [0,1).
	\end{equation}
	
	It follows from \eqref{tpi0} that $\tpi_0(\la)=0$ if and only if either $\Psi(1)=\Phi(1)=0$ or $\Psi(1)\neq 0$ and $\la=\frac{\Phi(1)}{\Psi(1)}$.
	However, the second case does not occur since we will show that $\frac{\Phi(1)}{\Psi(1)} \in \ov{\Lambda_{[0,1)}(\Delta)}$ which contradicts the assumption that $\la \notin \ov{\Lambda_{[0,1)}(\Delta)}$.  
	Indeed, it is not difficult to see that $\la_+(\Delta(t))$ has the asymptotic behaviour
	\[
	\la_+(\Delta(t)) = \frac{\Psi(1)}{\vartheta(1)}\frac{1}{(1-t)^2} + \o\Bigl(\frac{1}{(1-t)^2}\Bigr), \quad t\nearrow 1,
	\]
	and thus, for the other eigenvalue,
	\[
	\la_-(\Delta(t)) = \frac{\det(\Delta(t))}{\la_{+}(\Delta(t))} = \frac{1}{\la_{+}(\Delta(t))}\frac{\Phi(t)}{\vartheta(t)}\frac{1}{(1-t)^2} = \frac{\Phi(1)}{\Psi(1)} + \o(1),  \quad t\nearrow 1,
	\]	
	which yields that
	\[
	\frac{\Phi(1)}{\Psi(1)}\in \ov{\{\la_-(\Delta(t)):t\in[0,1)\}} \subset \ov{\Lambda_{[0,1)}(\Delta)}.
	\]
	This completes the proof of \eqref{tpi0=0}.

	Now suppose that $\tpi_0(\la)=0$, i.e.\ $\Psi(1)=0$, $\Phi(1)= 0$. Then it follows
 from \eqref{tpi1} that $\tpi_1(\la)=0$ if and only if either $\Psi'(1)=\Phi'(1)=0$ or $\Psi'(1)\neq0$ and $\la=\frac{\Phi'(1)}{\Psi'(1)}$. The latter cannot occur since we will show that $\frac{\Phi'(1)}{\Psi'(1)} \in \ov{\Lambda_{[0,1)}(\Delta)}$ which contradicts the assumption that $\la \notin \ov{\Lambda_{[0,1)}(\Delta)}$. 
Indeed, it is not difficult to see that, if $\Psi(1)=0$, $\Phi(1)= 0$, then $\la_+(\Delta(t))$ has the asymptotic behaviour
	\[
	\la_+(\Delta(t)) = -\frac{\Psi'(1)}{\vartheta(1)}\frac{1}{1-t} + \o\Bigl(\frac{1}{1-t}\Bigr),  \quad t\nearrow 1,
	\]
	and thus, for the other eigenvalue,
	\[
	\la_-(\Delta(t)) = \frac{1}{\la_+(\Delta(t))}\frac{\Phi(t)}{\vartheta(t)}\frac{1}{(1-t)^2} = \frac{\Phi'(1)}{\Psi'(1)} + \o(1),  \quad t\nearrow 1,
	\]
	which yields that
	\[
	\frac{\Phi'(1)}{\Psi'(1)}\in \ov{\{\la_-(\Delta(t)):t\in[0,1)\}} \subset \ov{\Lambda_{[0,1)}(\Delta)}.
	\]
It remains to be noted that
\begin{align}
\label{last1}
\Psi(1)=\Phi(1)=\Psi'(1)=0 \ \implies \ \Phi'(1)=0
\end{align}
by the definition of $\Phi$ in \eqref{Phi} since the three conditions on the left-hand side are equivalent to $(\vartheta \delta_{11}\!-\!\beta_1^2)(1) \!=\! (\vartheta \delta_{12}\!-\!\beta_1\beta_2)(1) \!=\! (\vartheta \delta_{11}\!-\!\beta_1^2)^{\prime}(1) \!=\! 0$.
This completes the proof of \eqref{tpi1=0} and hence of Proposition \ref{lem:pi.moeller}.
\end{proof}

Now we are ready to verify Assumption \ref{ass:suff} resp.\ \ref{ass:suff'}.
The conditions on $\mathcal{R}(\cdot,\la)$ in \eqref{ass:S.Rb} for Case (I) and in \eqref{rem:SII.Rb} for Case (II) are satisfied because $\tpi(\cdot,\la)\in C^1([0,1])$ and hence Taylor's theorem can be applied. In Case (III), we require the additional smoothness  assumptions 
	\begin{align}
	\label{ass:additional_in_III}
	\wt\vartheta,\wt\delta_{11},\wt\delta_{22},\wt\beta_1,\wt\beta_2\in C^2([0,1],\RR),
	\end{align}
which ensure that $\tpi(\cdot,\la)\in C^2([0,1])$ so $\mathcal{R}(\cdot,\la)$ satisfies \eqref{ass:S.Rb} again by Taylor's theorem.

Straightforward calculations yield
\begin{equation}
(D(t)-\la)^{-1}b(t)=\frac{-\ii}{\Xi(t,\la)}
\begin{pmatrix}
\ds(1-t)\bigl((\beta_1\delta_{22}-\beta_2\delta_{12})(t)-\la\beta_1(t)\bigr) \\[2ex]
\ds(\beta_2\delta_{11}-\beta_1\delta_{12})(t)-\la(1-t)^2\beta_2(t)
\end{pmatrix}, \quad t\in [0,1),
\end{equation}
and, because the condition $\tpi_0(\la)=0$ in \eqref{Cases_moeller}, \eqref{Cases_moellerc} implies that $(\beta_2\de_{11}-\beta_1\de_{12})(1)=0$, for
$t\nearrow 1$,
\begin{equation}
\label{moeller:Db}
\|(D(t)-\la)^{-1}b(t)\|^2_{\CC^n} = 
\begin{cases}
\O(1) &\text{in Case (I)},\\
\O((1-t)^2) &\text{in Cases (II), (III)}.
\end{cases}
\end{equation}
Hence the conditions for $\|(D(t)-\la)^{-1}b(t)\|^2_{\CC^n}$ in \eqref{ass:S.r} and \eqref{rem:SII.D} are satisfied.

To verify the conditions for $r(\cdot,\la)$ in \eqref{ass:S.r} and \eqref{rem:SII.r}, observe that the relation between $b$ and $c$, see \eqref{ex.b.coef}, and $D=D^*$ imply
\begin{align*}
r(\cdot,\la) = \Im (b^*(D-\la)^{-1}c) = 
\Im\bigl(b^*(D-\la)^{-1}(\gamma,0)^*\bigr)
= - \Im \big( (\gamma, 0) (D-\la)^{-1} b \big), 
\end{align*}
and thus, by \eqref{moeller:Db},
\begin{align}\label{r_moeller}
r(t,\la) = 
\frac{1-t}{\Xi(t,\la)}\Re\big(\gamma(t)\bigl((\beta_1\delta_{22}-\beta_2{\delta}_{12})(t)-\la\beta_1(t)\bigr)\big)=\O(1-t), \quad t\nearrow 1.
\end{align} 
confirming \eqref{ass:S.r} for Cases (I), (III) and \eqref{rem:SII.r} for Case (II). 

Finally, elementary calculations yield
\begin{equation}\label{kappa.Gamma}
\varkappa(t,\la) = \phi(t) -\la + \frac{1}{2(1-t)}\frac{\partial}{\partial 
	t}\tpi(t,\la) + \frac{3}{4}\frac{\tpi(t,\la)}{(1-t)^2} - 
\frac{|\gamma(t)|^2(\delta_{22}(t)-\la)}{\Xi(t,\la)}(1-t)^2 
-(1-t)\frac{\partial}{\partial t}\Gamma(t,\la)
\end{equation}
for $t\in [0,1)$ where
\begin{align}\label{Gamma}
\Gamma(t,\la) := 
\frac{1}{\Xi(t,\la)}\Im\bigl(\gamma(t)\bigl((\beta_1(\delta_{22}-\la)-\beta_2\delta_{12})(t)\bigr)\bigr).
\end{align}
Once again we use that $\la\notin \Lambda_\beta(D)$ implies that $\bigl(\delta_{11}\delta_{22}-|\delta_{12}|^2-\la\delta_{11}\bigr)(1) \neq 0$ and hence 
\vspace{-1.5mm} $\dfrac 1{\Xi(t,\la)}=\O(1)$, $t\nearrow 1$.
Therefore,
\[
\varkappa(t,\la) = \frac{3}{4}\frac{\tpi(t,\la)}{(1-t)^2} + 
\frac{1}{2} \frac{1}{1-t}\frac{\partial}{\partial t}\tpi(t,\la) + \O(1),  \quad t\nearrow 1,
\]
and thus
\[
s_\pi \varkappa(t,\la) = \begin{cases}
\ds\frac{3}{4}\frac{|\pi_0(\la)|}{(1-t)^2}(1+\o(1)) + \O(1) &\text{in Case 
	(I)},\\
\ds \frac 14 \frac{|\pi_1(\la)|}{1-t} (1+\o(1)) + \O(1) &\text{in Case (II)},\\
\ds \O(1) &\text{in Case (III)}.
\end{cases}
\]
Hence using \eqref{moeller:Db}, the relation between $b$ and $c$, see \eqref{ex.b.coef}, and
\begin{equation}
\left\|
(D(t)-\la)^{-1} 
\begin{pmatrix}
\,\ov{\gamma(t)}\, \\ 0
\end{pmatrix} 
\right\|^2_{\CC^n} 
= 
\O(1), \quad t\nearrow 1,
\end{equation}
we obtain that \eqref{ass:S.kappac} is satisfied in Cases (I), (III) and \eqref{rem:SII.kappac} is satisfied in Case (II).

\subsubsection*{Essential spectrum in Cases {\upshape(I)} and {\upshape(II)}}

By what was shown above, we can apply Theorem \ref{prop:suff} \ref{prop:suff.I} which yields that, in Cases (I) and (II), 
\[
\sesssing(\cA) \setminus \Lambda_\beta(D) = \emptyset,
\]
and hence, by Theorem \ref{thm:ess.abs},   
\begin{align*}
\sess(\wt\cA) \setminus \Lambda_{\beta}(D) = \sess(\cA) \setminus \Lambda_{\beta}(D) = \sessreg(\cA) \setminus \Lambda_{\beta}(D) = \ov{\Lambda_{[0,1)}(\Delta)} \setminus \Lambda_{\beta}(D) = \ov{\Lambda_{(0,1]}(\wt\Delta)}\setminus \Lambda_{\beta}(D),
\end{align*}
where $\wt\Delta$ is given by \eqref{moeller:Delta-tilde}.

\smallskip

\subsubsection*{Essential spectrum in Case {\upshape(III)}}

In this case, we first note that the additional smoothness assumptions \eqref{ass:additional_in_III} ensure
that Assumption \ref{ass:pres.reg} is satisfied. Next, we analyse the limits in Theorem \ref{prop:suff}~\ref{prop:suff.III}.
\begin{lemma}\label{lem:moeller_limits}
Let $\la \in \RR \setminus(\sessreg(\cA) \cup \Lambda_\beta(D))$. Then, in Case {\rm (III)}, the limits in \eqref{suff.lim.pi.r.kappa} exist and, 
using \linebreak $\la_{1,1} = \de_{22}(1) - \dfrac{|\de_{12}(1)|^2}{\de_{11}(1)}$  given in \vspace{-2mm}  \eqref{la_2,2}, 
\begin{align}
\label{tpi2lim}
\tpi_2(\la)  
&= \frac{2\vartheta(1)^2\la^2-\vartheta(1)\Psi''(1)\la+\bigl(\vartheta\Phi\bigr)''(1)}{2\beta_1(1)^2(\la_{1,1}-\la)} \in \RR \setminus \{0\}, 
\\
r_1(\la) &= -\frac{\beta_1(1)}{\delta_{11}(1)}\Re \gamma(1) \in \RR,\\
\label{ka0lim}
\varkappa_0(\la) &=  \phi(1) -\la -\frac{1}{4}\tpi_2(\la) \in \RR.
\end{align}
\end{lemma}

\begin{proof}
Let $\la\notin \sessreg(\cA)\cup \Lambda_{\beta}(D) = \ov{\Lambda_{[0,1)}(\Delta)}\cup \Lambda_{\beta}(D)$, where $\Delta$ is as in \eqref{moeller:Delta-tilde}. Throughout this proof we use that in Case (III) we have $\Psi(1)\!=\!\Psi'(1)\!=\!\Phi(1)=0$ and hence $\Phi'(1)\!=\!0$, see~\eqref{last1}; note that this implies $\Theta(1,\la)=0$, $\Theta'(1,\la)=0$, see \eqref{pi_moellera}.

Due to the additional smoothness assumptions \eqref{ass:additional_in_III}, the function $\Theta(\cdot,\la)$ defined in \eqref{pi_moellera} belongs to $C^2([0,1],\RR)$. Using \eqref{pi_moellera} and $\Phi(1)=\Phi'(1)=0$, we find that the following limit exists and satisfies
\[
\lim_{t\nearrow 1}\frac{\partial^2}{\partial t^2}\Theta(t,\la) 
= 2\vartheta(1)\la^2 - \Psi''(1)\la + \Phi''(1)
= 2\vartheta(1)\la^2 - \Psi''(1)\la + \frac{\bigl(\vartheta\Phi\bigr)''(1)}{\vartheta(1)} \in \RR.
\]
Since $\Theta(1,\la)=0$, $\Theta'(1,\la)=0$ and $\Xi(1,\la)\neq 0$ because $\la\notin\Lambda_\beta(D)$, it is not difficult to see that also the following limit exists and satisfies
\begin{align}\label{pi2.moeller}
\lim_{t\nearrow 1}\frac{\partial^2}{\partial t^2}\tpi(t,\la) = \lim_{t\nearrow 1}\frac{\partial^2}{\partial t^2}\frac{\Theta(t,\la)}{\Xi(t,\la)} = \lim_{t\nearrow 1}\frac{1}{\Xi(t,\la)}\frac{\partial^2}{\partial t^2}\Theta(t,\la) = \frac{2\vartheta(1)^2\la^2-\vartheta(1)\Psi''(1)\la+\bigl(\vartheta\Phi\bigr)''(1)}{\beta_1(1)^2(\la_{1,1}-\la)} \in \RR.
\end{align}
Moreover, L'H\^opital's rule yields 
\begin{align}\label{lim:tpi/(1-t)^2}
\lim_{t\nearrow 1}\frac{\tpi(t,\la)}{(1-t)^2} = \frac{1}{2} \lim_{t\nearrow 1}\frac{\partial^2}{\partial t^2}\tpi(t,\la) = \tpi_2(\la)
\end{align}
and thus Lemma \ref{lem:pidet}, together with \eqref{asymp:la_1}, implies that
\[
\lim_{t\nearrow 1}\det(\Delta(t)-\la) = \vartheta(1)\lim_{t\nearrow 1}\frac{\tpi(t,\la)}{(1-t)^2} \lim_{t\nearrow 1} \big( (1-t)^2\det(D(t)-\la) \big) = \tpi_2(\la)\vartheta(1)\delta_{11}(1)(\la_{1,1}-\la).
\]
Hence $\tpi_2(\la) \neq 0$, for otherwise $\lim_{t\nearrow 1}\det(\Delta(t)-\la)=0$, i.e.~$\la\in\ov{\Lambda_{[0,1)}(\Delta)}$, a contradiction to our assumption on $\la$.
Since $\Xi(1,\la) \neq 0$, it is easy to see from \eqref{r_moeller} that the following limit exists and satisfies
\begin{align*}
\lim_{t\nearrow 1}\frac{\partial}{\partial t}r(t,\la) 
= 
- \Re\Bigg\{
\lim_{t\nearrow 1}\frac{\gamma(t)\bigl(\beta_1(\delta_{22}-\la)-\beta_2\delta_{12}\bigr)(t)}{\Xi(t,\la)}
\Bigg\} 
= 
- \frac{\beta_1(1)}{\delta_{11}(1)}\Re \gamma(1) \in \RR;
\end{align*}
here, for the second equality, we used the relation
\begin{align}\label{csq1}
\bigl(\beta_1 \delta_{22}- \beta_2 \delta_{12}-\la \beta_1\bigr)(1)
= 
\frac{\beta_1(1)}{\delta_{11}(1)}\Xi(1,\la),
\end{align}
which is a simple consequence of the first two conditions of Case (III). 

The existence of the limit $\varkappa_0(\la)=\lim_{t\nearrow 1}\varkappa(t,\la)$ and the claimed formula for $\varkappa_0(\la)$ follow from \eqref{kappa.Gamma}, \eqref{Gamma}, the fact that $\Gamma(\cdot,\la) \in C^1([0,1])$, $\Xi(1,\la) \neq 0$, and \eqref{lim:tpi/(1-t)^2}. 
\end{proof}

Lemma \ref{lem:moeller_limits} guarantees that we can apply Theorem \ref{prop:suff} \ref{prop:suff.III} to calculate the singular part of the essential spectrum of any closed symmetric extension $\cA$ of the  operator $\cA_0$. For  $\la\notin\ov{\Lambda_{[0,1)}(\Delta)} \cup \Lambda_\beta(D)$,
\begin{align}
\label{last_ineq_0}  
\la \in \sesssing(\cA) \ 
\iff \ r_1(\la)^2-\varkappa_0(\la)\tpi_2(\la) \geq \frac{1}{4}\tpi_2(\la)^2 \iff \ (\la_{1,1}-\la)P(\la) \geq 0,
\end{align}
where
\begin{equation}
\begin{aligned}
P(\la) :=&  (\la-\phi(1))\Bigg(\la^2-\frac{1}{2}\frac{\Psi''(1)}{\vartheta(1)}\la+\frac{1}{2}\frac{(\vartheta\Phi)''(1)}{\vartheta(1)^2}\Bigg) -\bigl(\Re\gamma(1)\bigr)^2(\la-\la_{1,1})\\
\label{P(la)}
       =& (\la-\wt\phi(0))(\la^2-K_1\la+K_2) - \bigl(\Re\wt\gamma(0)\bigr)^2(\la-\la_{1,1})
\end{aligned}
\end{equation}
\vspace{-0.3mm}with
\begin{equation}
\begin{aligned}
\label{last4}
K_1 :=& \frac{1}{2\wt\vartheta(0)}(\wt\vartheta\wt\delta_{11}-\wt\beta_1^2)''(0) + \frac{1}{\wt\vartheta(0)}(\wt\vartheta\wt\delta_{22}-\wt\beta_2^2)(0),\\
K_2 :=& \frac{1}{2\wt\vartheta(0)^2}(\wt\vartheta\wt\delta_{11}-\wt\beta_1^2)''(0)(\wt\vartheta\wt\delta_{22}-\wt\beta_2^2)(0) - \frac{1}{\wt\vartheta(0)^2}|(\wt\vartheta\wt\delta_{12}-\wt\beta_1\wt\beta_2)'(0)|^2.
\end{aligned}
\end{equation}
Here we have used that $\wt\vartheta(0) \wt\de_{11}(0) = \wt\beta_1(0)^2$ by the first condition in Case (III).

Next we show that $\Lambda_\beta(D) = \{\la_{1,1}\}\in \sess(\cA)$. To this end, we consider the following possible cases:

\underline{Case 1:} Either $\la_{1,1}=\wt\phi(0)$, or $\la_{1,1}\neq\wt\phi(0)$ and $\la_{1,1}^2-K_1\la_{1,1}+K_2 \neq 0$.
We show that $\la_{1,1}$ is a limit point of the solution set $\Sigma$ of the inequality \eqref{last_ineq_0} and thus $\la_{1,1} \in \ov{\Sigma} \subset \sess(\cA)$, see Theorem \ref{singstr}. First assume $\la_{1,1}=\wt\phi(0)$. Then, since $\la\notin\Lambda_\beta(D) = \{\la_{1,1}\}$, the last inequality in \eqref{last_ineq_0} 
takes the~form 
\begin{align}
\label{last_ineq_1}
\la^2-K_1\la+K_2\leq \bigl(\Re\wt\gamma(0)\bigr)^2.
\end{align}
On the other hand, 
the first two conditions in Case (III) imply 
\begin{equation}
\label{last2}
\wt\vartheta(0) = \dfrac{\wt \beta_1(0)\wt \beta_2(0)}{\wt\de_{12}(0)}, \quad
\dfrac{\wt \beta_1(0)}{\wt \beta_2(0)}=\dfrac{\wt\de_{12}(0)}{\wt\de_{11}(0)}, \quad 
\la_{1,1} = \wt \de_{22}(0) - \frac{\wt\beta_2(0)^2}{\wt\vartheta(0)}.
\end{equation}
Hence it follows that
\begin{equation}
\label{last3}
\wt\vartheta(0)^2(\la_{1,1}^2-K_1\la_{1,1}+K_2) = - |(\wt\vartheta\wt\delta_{12}-\wt\beta_1\wt\beta_2)'(0)|^2 \leq0 \leq \bigl(\Re\wt\gamma(0)\bigr)^2,
\end{equation}
which shows that $\la_{1,1}$ satisfies \eqref{last_ineq_1} and so $\la_{1,1} \in \ov{\Sigma} \subset \sess(\cA)$. Secondly, assume that $\la_{1,1}\neq\wt\phi(0)$ and $\la_{1,1}^2-K_1\la_{1,1}+K_2 \neq 0$. 
Since $\la_{1,1}^2-K_1\la_{1,1}+K_2<0$ by \eqref{last3},
we find
\begin{align*}
	\lim_{\la\nearrow \la_{1,1}}(\la_{1,1}-\la)P(\la) = \infty \quad \text{if} \quad \la_{1,1}<\wt\phi(0), \qquad  \lim_{\la\searrow \la_{1,1}}(\la_{1,1}-\la)P(\la) = \infty \quad \text{if} \quad \la_{1,1}>\wt\phi(0).
\end{align*}
Therefore, in both cases, $\la_{1,1} \in \ov{\Sigma} \subset \sess(\cA)$.

\underline{Case 2:} $\la_{1,1}\neq\wt\phi(0)$ and $\la_{1,1}^2-K_1\la_{1,1}+K_2=0$. By \eqref{last3}, the latter yields that $(\wt\vartheta\wt\delta_{12}-\wt\beta_1\wt\beta_2)'(0)=0$. Thus it follows from L'H\^opital's rule that the matrix function $\wt\Delta$ given by \eqref{moeller:Delta-tilde} has a finite (componentwise) limit  as $t\searrow 0$. 
Using the formula for $\wt\Delta$ in \eqref{moeller:Delta-tilde}, the definition of $K_1$ in \eqref{last4}, and \eqref{last3}, we find
\begin{align*}
\ds\lim_{t\searrow 0}\wt\Delta(t) = 
\begin{pmatrix}
K_1-\la_{1,1} & 0 \\[1ex]
0             & \la_{1,1}
\end{pmatrix}.
\end{align*}
Hence $\la_{1,1} \in \ov{\Lambda_{(0,1]}(\wt\Delta)} = \sessreg(\cA)\subset \sess(\cA)$.

Altogether, we have now shown that 
\begin{equation}
\label{moeller:conclusion}
\sess(\wt\cA) = \sess(\cA) = \ov{\Lambda_{(0,1]}(\wt\Delta)}\, \cup \{\la\in\RR: \, (\la_{1,1}-\la)P(\la)\geq0\}, 
\quad \la_{1,1}=\wt\delta_{22}(0)-\frac{|\wt\delta_{12}(0)|^2}{\wt\delta_{11}(0)},
\end{equation}
where $P$ is the cubic polynomial given by \eqref{P(la)}.

\subsubsection*{The structure of $\sesssing(\cA)$ in Case {\rm (III)}} 
In order to analyse the structure of the singular part of the essential spectrum, we use our abstract Theorem \ref{singstr}. To this end, we need to verify Assumption (T) and compute some of the coefficients therein and in Proposition \ref{eshmat1}.

By what was shown above, the first part of Assumption (T1) holds with $j_0=1$; further, the estimate \eqref{ass:max.min} holds trivially since $j_0=1$ and $n=2$. By Lemma \ref{lem:moeller_limits}, Assumptions (T2), (T3) are satisfied and hence Proposition \ref{eshmat1} applies. Now the formulas \eqref{tpi2lim} for $\tpi_2(\la)$ and \eqref{ka0lim} for $\kappa_0(\la)$ in Lemma \ref{lem:moeller_limits} yield that
	\[
	g_\beta = \lim_{\la\to \infty} \frac {-\tpi_2(\la)}{\la} = \frac{\vartheta(1)^2}{\beta_1(1)^2} > 0,
	\quad
	\psi_\beta = 1 - \frac 14 g_\beta, \quad g_\beta+ 4\psi_\beta = 4 > 0.  
	\]
This means we are in the first case of Theorem \ref{singstr} (a) and hence the set $\ov{\Sigma}$ consists of the union of at most two compact intervals.
	
Since  $\wt\Delta \in C^1((0,1],\CC^{2\times 2})$ has a limit as $t\searrow 0$, the closure of the range of its eigenvalues $\ov{\Lambda_{[0,1)}(\Delta)} \!\subset\! \RR$ has at most two components; hence also the regular part of the essential spectrum is the union of at most two compact intervals.

Moreover, it is not difficult to see that, in Case (III), the eigenvalues $\la_{\pm}(\wt\Delta(t))$ of $\wt\Delta(t)$ have the asymptotic expansions
	\begin{align}
	\label{moeller:asymp.Delta.III}
	\la_+(\wt\Delta(t)) = \frac{K_1}{2}+\sqrt{\frac{K_1^2}{4}-K_2} + \o(1), \quad \la_-(\wt\Delta(t)) = \frac{2K_2}{K_1+\sqrt{K_1^2-4K_2}} + \o(1), \quad t\searrow 0.
	\end{align}
Hence both limits $\lim_{t\searrow 0}\la_{\pm}(\wt\Delta(t))$ satisfy the inequality $(\la_{1,1}-\la)P(\la)\geq0$ and thus belong to $\sessreg(\cA)$ and ~$\ov{\Sigma}$.
Altogether, we conclude that $\sess(\wt\cA)$ is union of at most two compact intervals.

\subsubsection*{Essential spectral radius}

In Cases (I) and (II), the functions $\la_{\pm}(\Delta(\cdot))$ in \eqref{lapmDelta} are not bounded and hence $r_{\rm{ess}}(\wt\cA)\!=\! \infty$. In Case (III), $r_{\rm{ess}}(\wt\cA)\!<\! \infty$ since we have shown that $\sess(\wt\cA)$ is the union of at most two com\-pact intervals. 
An explicit formula for $r_{\rm{ess}}(\wt\cA)$ may be given by finding the root with largest absolute value of the cubic polynomial $P$ in \eqref{P(la)}, e.g.\ by means of Cardano's formula; we refrain from displaying the  elementary, but too lengthy, formulas here.

\begin{remark}
Our abstract results give a new proof for the results of the paper of \cite{MNT02} and of the observation noted therein that the regular and singular part of the essential spectrum are adjoined to each other.
\end{remark}


\section{Application to a spectral problem for symmetric stellar equilibrium models}

In this section, we investigate a matrix differential operator arising in the stability analysis of spherically symmetric stellar equilibrium models, see \cite[Section 4.1]{BEY95} and \cite{ILLT13}. 
This operator represents the unperturbed part of the reduced spheroidal operator in the radial variable $t\in (0,R)$, where $R>1$ is the radius of the star, for polytropic equilibrium models with constant adiabatic index near the centre $(t=0)$ and near the boundary $(t=R)$ of the star. 

Since the coefficient functions have singularities at both end-points $t=0$ and $t=R$, Glazman's decomposition principle was used in \cite{ILLT13} to split the essential spectrum into the essential spectra of the corresponding operators $\cA_{(0,1]}$ and $\cA_{[1,R)}$ restricted to the intervals $(0,1]$ and $[1,R)$. It was proved in  \cite{ILLT13} that, for both operators, the regular part of the essential spectrum is only the single point $\{0\}$
and that the singular part of the essential spectrum of the operator on $(0,1]$ is~empty,
\[
 \sessreg(\cA_{(0,1]}) = \sessreg(\cA_{[1,R)}) = \{0\}, \quad \sesssing(\cA_{(0,1]})=\emptyset.
\]

However, the method of \cite{ILLT13} could not be used to determine the singular part of the essential spectrum of the operator on $[1,R
)$. The reason for this is that the first derivative of the Lane-Emden function $\theta=\theta_n$ entering the coefficient functions, see \eqref{reallylast}, \eqref{prho} below, does not vanish at $R$ and hence the coefficients of the Schur complement are not bounded at~$R$. It was  conjectured in \cite{ILLT13} that, nevertheless, the singular part of the essential spectrum on $[1,R)$ is empty as well.

We prove this conjecture and thus show that the essential spectrum of every self-adjoint extension of the operator on the full interval $(0,R)$ consists only of the single point $\{0\}$.

\subsection*{Example from astrophysics}

In the Hilbert space $L^2((1,R)) \oplus L^2((1,R))$ we
consider the operator  matrix
\begin{align}\label{A_0_astro}
 \cA_0 = \begin{pmatrix} -\dt p_1 \dt + q_1 &  \ds \dt p_2+q_2 \\[1.8ex]
  \ds -p_2\dt+q_2                               &   p_3 \end{pmatrix}
\end{align}
with domain $\cD(\cA_0)=C_0^{2}((1,R)) \oplus C_0^{1}((1,R))$.
Using the notation of \cite{BEY95}, 
the coefficient functions $p_1, p_2,$ $p_3, q_2 \in C^1([1,R], \RR)$, and $q_1 \in C([1,R], \RR)$  
are given by
\begin{equation}
\label{reallylast}
\begin{array}{l}
\displaystyle p_1:= \frac{\Gamma_1 p}{\varrho}, \quad p_2:= c_l \frac{\Gamma_1p}{t\varrho}, \quad p_3:=c_l^2\frac{\Gamma_1p}{t^2\varrho},\\[-4mm]
\displaystyle q_1:= \frac{1}{t\varrho}\bigl((4-3\Gamma_1)p\bigr)'+\frac{1}{t^2\sqrt{\varrho}}\Bigl(\frac{\Gamma_1p}{\varrho}(t^2\sqrt{\varrho})'\Bigr)', 
\quad q_2:= \overbrace{c_l\frac{\Gamma_1p}{t\varrho}}^{=p_2}\Bigl(\frac{\varrho'}{\varrho}-\frac{p'}{\Gamma_1p}-\frac{1}{2}\frac{(t^2\varrho)'}{t^2\varrho}\Bigr),
\end{array}
\end{equation}
where the constant $c_l=\sqrt{l(l+1)}$, $l \in \N$, appears after the reduction of the problem in $\RR^3$ to the radial part 
and the coefficient functions $\Gamma_1$, $p$, and $\rho$ represent the following physical quantities. 
%
The function $\Gamma_1\!\in\! C^{2}([1,R],\RR)$ is the adiabatic index which is positive on $[1,R]$ and satisfies $\Gamma_1'(R)=0$.
The functions $p\in C^2([1,R],\RR)$ and $\varrho\in C^3([1,R],\RR)$ are the pressure and mass density, respectively. They are both positive on $[1,R]$ and are supposed to have the~forms
\begin{align}
\label{prho}
p(t)=p_c \theta_n(t)^{n+1},\quad \varrho(t)=\varrho_c \theta_n(t)^n, \quad t\in [1,R],
\end{align}
where $p_c$, $\varrho_c >0$ are the constant central pressure and central density, respectively, of the unperturbed star and 
$n \in (0,5)$ is the polytropic index; here the physically most interesting case is $n \in [1,5)$, see \cite[Section~5, p.\ 47]{BEY95}.
The function $\theta_n \in C^2([1,R],\RR)$ is the Lane-Emden function of index $n$ which is uniformly positive on $[1,R)$ and satisfies the non-linear differential equation 
\begin{align}
\label{L-M1}
\theta''_n(t)+\frac{2}{t}\theta'_n(t)=-\frac 1{\alpha_n^2} (\theta_n(t))^n, \quad t\in (0,\infty),
\end{align}  
where $\alpha_n\in (0,\infty)$ is the Lane--Emden unit length and $R=R_n >1$ is the first zero of $\theta_n$, see~\cite{BEY95,CHAN57}.

\subsubsection*{Transformation to the form \eqref{A02} and verification of 
Assumption \ref{ass:regularity}}

Note that \eqref{A_0_astro} is already of the form \eqref{A02} with $\alpha=1$, $\beta=R$, $n=1$, $p=p_1$, $q=q_1$, $b=-p_2$, $c=q_2$, and $D=p_3$. Assumption ~\ref{ass:regularity} is clearly satisfied due to the smoothness assumptions on the coefficient functions $p_1$, $p_2$, $p_3$, and $q_1$,~$q_2$.

\subsubsection*{The set $\Lambda_\beta(D)$} Since $R$ is the first zero of $\theta_{n}$, \eqref{prho} yields
\[
\lim_{t\nearrow R}D(t) = \lim_{t\nearrow R} p_3(t) = \frac{c_l^2}{R^2}\frac{p_c}{\varrho_c}\Gamma_1(R)\lim_{t\nearrow R}\theta_n(t) =0
\]
and thus $\Lambda_\beta(D) = \{0\}$. Note that, since $p_1p_3 \equiv p_2^2$ on $[1,R]$, we have 
$\Delta(t) = 0$, $t\in [1,R)$, and hence
$\sessreg(\cA) = \{0\} =\Lambda_\beta(D)$ for every closed symmetric extension $\cA$ of $\cA_0$ in $L^2((1,R)) \oplus L^2((1,R))$.

\subsubsection*{Verification of Assumption \ref{ass:suff} resp.\ {\rm (S')}}

Elementary calculations show that $\tpi(\cdot,\la)$ has an asymptotic expansion \eqref{pi_asymp} as $t\nearrow R$  with 
\[
 \tpi_0(\la) = 0, \quad \tpi_1(\la) = \frac{p_c}{\varrho_c}\Gamma_1(R)\theta'(R), \quad \la\in \RR \setminus \{0\}.
\]  

\begin{lemma}
\label{lemma_astro_1}
The Lane-Emden function $\theta_n$ satisfies 
\begin{align}
\label{astro:lemma}
\lim_{t\nearrow R}(t-R)\frac{\theta_n'(t)}{\theta_n(t)} = 1.
\end{align}
\end{lemma}

\begin{proof}
First of all, note that $\theta_n'(R) \neq 0$, for otherwise, since $\theta_n(R)=0$, the theorem on the uniqueness of solution of ODEs with continuous coefficients on closed intervals would imply that $\theta_n(t) \equiv 0$ on $[1,R]$, contradicting to uniform positivity of $\theta_{n}$ on $[1,R)$. Taylor's formula with remainder term of Lagrange form yields, for some $\gamma\in (0,1)$,
\[
\theta_n(t) = \theta_n(R) + (t-R)\theta_n'(R) + \frac{(t-R)^2}{2}\theta_n''(R+\gamma(t-R)).
\]
Since $\theta_n(R)=0$, we obtain \eqref{astro:lemma}.
\end{proof}

Note that, since $\theta_n$ is uniformly positive, the above lemma implies $\theta_n'(R)<0$ and thus $\tpi_1(\la)<0$. So we are in Case (II) and it suffices to verify Assumption (S'), i.e.\ conditions \eqref{rem:SII.Rb}--\eqref{rem:SII.kappac}, see Remark \ref{rem:SII}. 

Since $\tpi(\cdot,\la)\in C^1([1,R])$, the condition for $\mathcal{R}(\cdot,\la)$ in \eqref{rem:SII.Rb} is satisfied by Taylor's theorem. Moreover, $r(\cdot,\la) = -\Im (p_2q_2/(p_3-\la)) \equiv 0$ as all coefficients functions are real-valued and so \eqref{rem:SII.r} is trivially~satisfied. 

Elementary calculations, together with Lemma \ref{lemma_astro_1}, yield
\[
\|(D(t)-\la)^{-1}b(t)\|^2 = \Bigl|\frac{p_2(t)}{p_3(t)-\la}\Bigr|^2 = \frac{c_l^2}{\la^2}\frac{p_c^2}{\varrho_c^2}\frac{\Gamma_1(R)^2}{R^2}\theta_n'(R)^2 (R-t)^2 + \o((R-t)^2), \quad t\nearrow R,
\]
and also
\[
\frac{q_2(t)}{p_2(t)} = \Bigl(\frac{n}{2}-\frac{n+1}{\Gamma_1(R)}\Bigr)\frac{1}{t-R} + \o\Bigl(\frac{1}{t-R}\Bigr),  \quad t\nearrow R.
\]
Hence
\[
\|(D(t)-\la)^{-1}c(t)\|^2 = \Bigl|\frac{q_2(t)}{p_2(t)}\Bigr|^2 \|(D(t)-\la)^{-1}b(t)\|^2 = \O(1), \quad t\nearrow R.
\]
Moreover, 
\[
\varkappa(t,\la) = \frac{n^2}{4}\frac{p_c}{\varrho_c}\frac{(\theta'(t))^2}{\theta(t)}\Gamma_1(t) + \O(1), \quad t\nearrow R.
\]
Since $s_{\tpi}=1$, $\Gamma_1(t)>0$, $\theta_n(t)>0$, $t\in[1,R)$ we have $(s_{\pi}\varkappa(\cdot,\la))_- = \O(1)$ so condition \eqref{rem:SII.kappac} is satisfied,~too. 

\subsubsection*{Essential spectrum}

By what was shown above, Theorem \ref{prop:suff} \ref{prop:suff.III} applies and, together with $\sessreg(\cA_{[1,R)})\!=\!\{0\}$ $=\Lambda_{\beta}(D)$, we obtain
\[
 \sesssing(\cA_{[1,R)}) = \emptyset, \quad \sess(\cA_{[1,R)}) = \sessreg(\cA_{[1,R)}) = \{0\}.
\]
for every closed symmetric extension $\cA_{[1,R)}$ of $\cA_0$ in $L^2((1,R)) \oplus L^2((1,R))$. This proves the conjecture in \cite{ILLT13} and yields, finally, 
that
\[
 \sess(\cA) = \sessreg(\cA) = \{0\} 
\]
for every closed symmetric extension $\cA$ of $\cA_0$ in $L^2((0,R)) \oplus L^2((0,R))$. 
\subsubsection*{Essential spectral radius}
Having proved the conjecture above, we can now conclude that $r_{\rm{ess}}(\wt\cA)=0$.

\pagebreak


\bigskip

\noindent
{\bf Acknowledgements.} \ 
The authors thank the referee for valuable comments and they
gratefully acknowledge the support of the \emph{Swiss National Science Foundation}, SNF, grant no.\ $200020\_146477$ (O.\,O.\ Ibrogimov and C.\ Tretter) as well as Ambizione grant no.\ PZ00P2$\_154786$ (P.\ Siegl). C.\ Tretter also thanks the \emph{Knut och Alice Wal\-lenbergs Stiftelse}, Sweden, for a guest professorship and the Matematiska institutionen at Stockholms universitet for the kind hospitality.

{\small
	\bibliographystyle{acm}
	\bibliography{ist}
}

\end{document}